\documentclass{article}
\usepackage[utf8]{inputenc}
\usepackage[full]{textcomp}
\usepackage[osf]{newtxtext}
\usepackage{comment}

\usepackage{amssymb}
\usepackage{mathtools}
\usepackage[colorlinks,linkcolor=blue,citecolor=blue,urlcolor=blue]{hyperref}
\usepackage{breakurl}
\usepackage{mhenvs}
\usepackage{mhequ} 
\usepackage{mhsymb}
\usepackage{booktabs}
\usepackage{tikz}
\usepackage{mathrsfs}
\usepackage{longtable}
\usepackage{times}
\usepackage{ulem}
\usepackage{braket}
\usepackage{xparse}
\usepackage[hang]{footmisc}

\usepackage[capitalise]{cleveref}

\newtheorem*{lemma**}{Lemma}
\newtheorem*{theorem**}{Theorem}

\theoremstyle{definition}

\usepackage{microtype}
\usepackage{wasysym}
\usepackage{centernot}
\usepackage{enumitem}

\numberwithin{equation}{section}

\makeatletter
\newcommand{\globalcolor}[1]{%
	\color{#1}\global\let\default@color\current@color
}
\makeatother

\newif\ifdark
\IfFileExists{dark}{\darktrue}{\darkfalse}

\ifdark

\definecolor{darkred}{rgb}{0.9,0.2,0.2}
\definecolor{darkblue}{rgb}{0.7,0.3,1}
\definecolor{darkgreen}{rgb}{0.1,0.9,0.1}
\definecolor{pagebackground}{rgb}{.15,.21,.18}
\definecolor{pageforeground}{rgb}{.84,.84,.85}
\pagecolor{pagebackground}
\AtBeginDocument{\globalcolor{pageforeground}}

\else

\definecolor{darkred}{rgb}{0.7,0.1,0.1}
\definecolor{darkblue}{rgb}{0.4,0.1,0.8}
\definecolor{darkgreen}{rgb}{0.1,0.7,0.1}
\definecolor{pagebackground}{rgb}{1,1,1}
\definecolor{pageforeground}{rgb}{0,0,0}

\fi

\theoremnumbering{Alph}
\renewtheorem{theorem*}{Theorem}

\DeclareMathAlphabet{\mathbbm}{U}{bbm}{m}{n}
\DeclareFontFamily{U}{BOONDOX-calo}{\skewchar\font=45 }
\DeclareFontShape{U}{BOONDOX-calo}{m}{n}{
	<-> s*[1.05] BOONDOX-r-calo}{}
\DeclareFontShape{U}{BOONDOX-calo}{b}{n}{
	<-> s*[1.05] BOONDOX-b-calo}{}
\DeclareMathAlphabet{\mcb}{U}{BOONDOX-calo}{m}{n}
\SetMathAlphabet{\mcb}{bold}{U}{BOONDOX-calo}{b}{n}

\let\epsilon\varepsilon
\def\E{{\symb E}}
\def\F{{\mathbf F}}
\def\FF{{\mathbb F}}
\def\H{{\mathscr H}}
\def\FC{\mathscr{C}}
\def\FD{\mathscr{D}}

\def\L{\mathbb L}

\def\X{\mathbb{X}}
\def\W{\mathbf{W}}

\def\WW{{ \mathbb W}}

\def\X{{\mathbf X}}
\def\XX{{\mathbb X}}

\def\Y{{\mathbf Y}}
\def\YY{{\mathbb Y}}

\def\D{{\mathcal D}}

\newcommand{\define}{\stackrel{\mbox{\tiny \rm def}}{=}}
\newcommand{\vol}{\mathcal{V}}
\newcommand{\expec}[1]{\E[#1]}
\newcommand{\Expec}[1]{\E\left[#1\right]}
\newcommand{\Sym}{\mathrm{Sym}}
\newcommand{\bl}{{\boldsymbol{\ell}}}
\newcommand{\G}{\mathbb{G}}
\newcommand{\bk}{\boldsymbol{k}}
\newcommand{\bm}{\boldsymbol{m}}

\newcommand{\herm}{\mathsf{H}}

\newcommand{\Tr}{\operatorname{{\mathrm Tr} }}

\DeclareMathOperator{\id}{id}

\renewcommand{\geq}{\geqslant}
\renewcommand{\leq}{\leqslant}

\usepackage{cleveref}

\def\C{\mathcal{C}}
\def\f{\frac}

\def\1{\mathbf{1}}

\def\${|\!|\!|}

\def\id{\mathrm{id}}

\def\<{\langle}
\def\>{\rangle}
\setlist{noitemsep,topsep=4pt}
\def\para_#1{/\!\!/_{\!#1}}

\newcommand{\fdd}{\xrightarrow{\textup{f.d.d.}}}

\def\slash{\kern0.18em/\penalty\exhyphenpenalty\kern0.18em}
\def\dash{\kern0.18em--\penalty\exhyphenpenalty\kern0.18em}

\newcommand{\vertiii}[1]{{\left\vert\kern-0.25ex\left\vert\kern-0.25ex\left\vert #1 \right\vert\kern-0.25ex\right\vert\kern-0.25ex\right\vert}}
\NewDocumentCommand{\Lin}{om}{\IfNoValueTF{#1}{\L(\R^{#2},\R^{#2})}{\L(\R^{#1},\R^{#2})}}

\makeatletter 
\newcommand*{\fat}{}
\DeclareRobustCommand*{\fat}{%
	\mathbin{\mathpalette\bigcdot@{}}}
\newcommand*{\bigcdot@scalefactor}{.5}
\newcommand*{\bigcdot@widthfactor}{1.15}
\newcommand*{\bigcdot@}[2]{%
	\sbox0{$#1\vcenter{}$}
	\sbox2{$#1\cdot\m@th$}%
	\hbox to \bigcdot@widthfactor\wd2{%
		\hfil
		\raise\ht0\hbox{%
			\scalebox{\bigcdot@scalefactor}{%
				\lower\ht0\hbox{$#1\bullet\m@th$}%
			}%
		}%
		\hfil
	}%
}
\makeatother
{\theorembodyfont{\rmfamily}
	\newtheorem{example}{Example}
	
}

\newtheorem{fact}[lemma]{Fact}
\newtheorem{condition}[lemma]{Condition}

\makeatletter
\newcommand\blfootnote[1]{%
	\begingroup
	\renewcommand\thefootnote{}\footnote{#1}%
	\addtocounter{footnote}{-1}%
	\endgroup
}
\makeatother

\renewcommand{\labelitemi}{$\triangleright$}

\begin{document}
\title{\vspace{-2cm}
Functional Limit Theorems for  Volterra Processes and Applications to Homogenization}
\author{Johann Gehringer, Xue-Mei Li, and Julian Sieber \\ Imperial College London}

\maketitle

\begin{abstract}
\noindent
We prove an enhanced limit theorem for additive functionals of a multi-dimensional Volterra process $(y_t)_{t\geq 0}$ in the rough path topology. As an application, we establish weak convergence as $\varepsilon\to 0$ of the solution of the random ordinary differential equation (ODE) $\frac{d}{dt}x^\varepsilon_t=\f 1  {\sqrt \epsilon}  f(x_t^\varepsilon,y_{\frac{t}{\varepsilon}})$ and show that its limit solves a rough differential equation driven by a Gaussian field with a drift coming from the L\'evy area correction of the limiting rough driver. Furthermore, we prove that the stochastic flows of the random ODE converge to those of the Kunita type It\^o SDE $dx_t=G(x_t,dt)$, where $G(x,t)$ is a semi-martingale with spatial parameters.
\end{abstract}

\blfootnote{Johann Gehringer is supported by an EPSRC studentship, Xue-Mei Li by the EPRSC grants EP/V026100/1 
and EP/S023925/1, and Julian Sieber by the EPSRC Centre for Doctoral Training in Mathematics of Random Systems: Analysis, Modelling and Simulation (EP/S023925/1).}


{\noindent\scriptsize \textit{MSC2020 Subject classification:} 34F05, 60F05, 60F17}

\setcounter{tocdepth}{2}
{\hypersetup{hidelinks}
\tableofcontents
}
\newpage
\section{Introduction}

The aim of this article is to obtain an approximate effective dynamics for the evolution of a particle moving  in a  fast oscillating non-Markovian random vector field  $f(\cdot, y_t(\omega))$. More precisely, we study the small-$\epsilon$ limit of the $\R^d$-valued solutions of the equation
\begin{equation}\label{eq:odenoisy}
 dx^\varepsilon_t= \f 1  {\sqrt \epsilon}  f(x_t^\varepsilon,y_{\frac{t}{\varepsilon}})\,dt.
\end{equation}
We begin with studying the convergence of stochastic processes of the form
$$
X^{\epsilon}_t=  \left( \sqrt{\epsilon} \int_0^{\frac{t}{\epsilon}} G_1(y_s) \,ds , \dots, \sqrt{\epsilon} \int_0^{\frac{t}{\epsilon}} G_N(y_s) \,ds\right)
$$
together with their canonical lifts $\XX_{s,t}^\epsilon\define \int_s^t (X^\epsilon_{r}-X^\epsilon_s)\,\otimes dX_r^\epsilon$, where  $(y_t)_{t\geq 0}$ is a multi-dimensional Volterra process with power law correlation decay $t^{-\beta}$ and spectral density. We obtain a functional limit theorem together with an enhanced and a rough CLT. 
The set of admissible functions $G_i$ for the CLT  are $L^2$ functions with Hermite rank bounded from below by $\frac{1}{\beta}$ (or by $\frac{2}{\beta}$ for the rough CLT) and rapidly decaying Hermite coefficients.

We then obtain a rough CLT for the function space valued stochastic process
$$
X^{\epsilon}_t= \sqrt{\epsilon} \int_0^{\frac{t}{\epsilon}} f(\cdot, y_s) \,ds=  \left( \sqrt{\epsilon} \int_0^{\frac{t}{\epsilon}} f_1(\cdot, y_s) \,ds , \dots, \sqrt{\epsilon} \int_0^{\frac{t}{\epsilon}} f_d(\cdot, y_s) \,ds \right)
$$ 
 and cast \eqref{eq:odenoisy} as a rough differential equation (RDE) in a Banach space driven by $(X^\varepsilon_t)_{t\in[0,T]}$. The convergence of the solution of \eqref{eq:odenoisy} is then an easy consequence of the continuous dependence of the RDE on its driver. We emphasize that---even when $y$ is a strong mixing Markov process so we are in the classical domain---it is advantageous to consider the ODE as an RDE. In fact with the approach presented here, we automatically obtain simultaneous convergence of the solutions with any finite number of initial conditions to the solutions of the same limiting SDE of Kunita type, which is not easy to prove with the usual martingale method. 

To illustrate the mechanism behind the homogenization problem, we first review the following popular perturbation model:
$$\dot z_t^\epsilon =\epsilon g( z_t^\epsilon, y_t)+ f(z_t^\epsilon, y_t),$$
where $f, g: \R^d\times\R^n\to \R^d$ are sufficiently regular with suitable initial conditions. On the scale of $[0, \f  1{\sqrt \epsilon}]$, $z_t^\epsilon$ can be identified with the solution of
the slow/fast system:
$$\dot x_t^\epsilon = g( x_t^\epsilon, y_{\f t \epsilon})+ \f 1 {\sqrt \epsilon} f(x_t^\epsilon, y_{\f t \epsilon}),  \qquad  x_0^\epsilon=x_0, \qquad x_0\in \R^d.$$
The parameter $\epsilon$ is a positive number tuning the relative speed with respect to the fast motion $y$ and is assumed to be small.
If the stochastic process $(y_t)_{t\geq 0}$ exhibits oscillatory features and $f$ averages to zero with respect to a measure determined by it,  the slow motion feels the averaged and the central limit theorem velocity.
It is then plausible to deduce a simpler equation, called the effective equation, whose solutions approximate $(z_t^\epsilon)_{t\geq 0}$ on the scale $[0, \f 1 {\sqrt \epsilon}]$.  This is the principle of slow/fast homogenization. Unlike for $(x_t^\epsilon)_{t\in [0,T]}$, the effective equation for $\lim_{\epsilon\to 0} x_t^\epsilon$ is autonomous. There are two well known theories on this: (1) the fast motion is deterministic, (2) the fast motion is strong mixing Markovian. If the fast motion is periodic this is classic. More recently,  diffusive homogenization  is obtained  for  a class of non-uniformly hyperbolic  fast flows including solutions of the classical Lorenz equations \cite{Kelly-Melbourne}. 
The second theory, in which $(y_t)_{t\geq 0}$ is an ergodic Markov process satisfying a strong mixing condition, is by now also classic \cite{Kesten-Papanicolaou, Freidlin-Wentzell}. 
The theory we intend to study is a fast dynamics with slowly decaying correlations. Such dynamics appear naturally in natural, economical, and social sciences, as well as in statistics, see e.g. \cite{Comte-Renault96, Rio17, Samorodnitsky, Robinson}.

In this article, $(y_t)_{t\geq 0}$ is taken to be an admissible `moving average' Gaussian process. Let $W_t$ be a Brownian motion in $\R^d$,
and $K(r)\in \L(\R^d, \R^n)$  an admissible kernel.
A Volterra process is of the form
\begin{equation}\label{eq:kernel_representation-1} 
y_t=\int_{-\infty}^t K(t-u)dW_u.
\end{equation}
Increments of (fractional) Brownian motions and the fractional Ornstein-Uhlenbeck processes are Volterra processes.   In fact, a centered one-dimensional stationary $L^2$ continuous Gaussian process $(y_t)_{t\geq 0}$ has the integral representation \eqref{eq:kernel_representation-1} if and only if its spectral measure has a density \cite[Prop. 14.20]{Kallenberg3} by simply taking $K$  the Fourier transform of the square root of the density. We will assume an algebraic correlation decay: for some $\beta>0$, $s,t\geq 0$,
$$\big|\Expec{y_s\otimes y_t}\big|
=  \left|\int_\R K(s-u)  K^*(t-u)\,du\right|  \leq \Theta \left( 1 \wedge |t-s|^{- \beta} \right).$$ 
For a stationary Gaussian processes, the  regularity of the spectral measure can be obtained from the decay rate of the correlation function.  In comparison,  `strong mixing',  which can be characterized by the spectral density, cannot be classified by the decay  rate of its correlation function.

Before going ahead explaining how to propose the convergence theorem, we go back to the well known situation in which  $y_t$  is a Markov process or strong mixing and take $g=0$ for simplicity.  Then the ansatz is a diffusion process whose generator can be formally written down and the road map for the convergence of  $x_t^\epsilon$  is  to show that $(x^\epsilon)_{\varepsilon\in(0,1]}$ is relatively compact followed by an application of the martingale problem method.  The paramount question is then to determine for which $L^2$ functions $G$ a functional central limit theorem  holds for $ {\sqrt \epsilon}  \int_0^{t/\epsilon} G(y_{s}) ds$. This is the Kipnis-Varadhan theory \cite{Kipnis-Varadhan} which has been become a cornerstone in studying  both continuous and discrete stochastic models, see \cite{Komorowski-Landim-Olla} and the references therein. We would like to highlight the fruitful rough path approach explored in \cite{Kelly-Melbourne, Gehringer-Li-homo, Deuschel-Orenshtein-Perkowski}.

In case $(y_t)_{t\geq 0}$ is a non-strong mixing Volterra Gaussian noise,  how do we formulate the condition on the driver?
Let us first take $(y_t)_{t\geq 0}$ to be the one-dimensional stationary fractional Ornstein-Uhlenbeck process
and consider  $$\alpha(G, \epsilon) \int_0^{t/\epsilon} G(y_{s}) \,ds.$$
Then the scale $\alpha(G, \epsilon) $ and the scaling limit depend on its Hermit rank, which is the lowest $m$ in the Hermite polynomial expansion of  $G=\sum_{i=m}^\infty c_iH_i$.  Furthermore, the solutions of the equations $\dot x_t^\epsilon=\sum_{i=1}^N\alpha(G_i, \epsilon) f_i(x_t^\epsilon)G_i( y_{t/\epsilon})$ 
converge  \cite{Gehringer-Li-tagged}. For $N=1$, this is straightforward. For $N>1$, this models different driving vector fields across different regions, is already more complex and involves CLT's for the iterated integrals in a topology stronger than the weak topology.  An enhanced CLT is deduced from and built upon the vast existing literature, c.f. \cite{Taqqu, Maejima-Ciprian, Campese-Nourdin-Nualart} under a fast chaos decay condition. The effective equation describes in some cases a diffusion, in other cases an anomalous super diffusion, or turns out to be a stochastic differential equation with mixed It\^o, Lebesgue, and Young integrals. For related work see \cite{Ivanov-Leonenko-Ruiz-Medina-Savich}.

The result we described applies to the product from $f(x,y)=\sum_{i=1}^N f_i(x)G(y)$ and the one-dimensional toy model noise.
The problem remains open for  the evolution in the random field generated by multi-dimensional noise and for the non-product form.
 For a preliminary examination, let us assume that the random field $F(x,t, \epsilon, \omega)=f(x, y_{\f t\epsilon(\omega)})$ is in $L^2$. Taylor expanding $F$ in $x$ in some direction $v$, we see
$$F(x,t, \epsilon, \omega)\sim f(x_0, y_{\f t \epsilon}(\omega))+ \nabla_x f (x_0, y_{\f t \epsilon}(\omega))\cdot  v+\dots.$$
The CLT applies only if $f$ is sufficiently fast oscillatory so the oscillation compensates the insufficient decay in the auto-correlation of $(y_t)_{t\geq 0}$.
 It is more efficient to take the Hermite expansion of $f$ with respect to the Gaussian process $(y_t)_{t\geq 0}$.
Wrapping the Hermite polynomials around a Gaussian process amplifies its correlation decay. We use this to determine the decay rate of the random field.
In case of $\E[y_0\otimes y_0]=\id$, the Hermite rank of $f(x,\cdot)$ is  the lowest $|\bl |$ with non-vanishing $c_{\bl}$ in the following Hermite polynomial expansion:
\begin{equation}\label{eq:hermite_intro}
f(x, y_t)=\sum_{\bl\in\N_0^n}c_{\bl}(x)H_{\bl}(y_t)\qquad \forall\,t\geq 0,
\end{equation}
 where $\bl=(\bl_1, \dots, \bl_n)$ is a multi-index and $c_{\bl}=(c_{\bl}^1, \dots, c_{\bl}^n)\in\R^n$. Otherwise we use a transformed process $(z_t)_{t\geq 0}$ for defining the Hermite rank.
The Hermite rank precisely characterizes whether  the functional CLT holds. We will specify the regularity conditions later. 

\medskip

{\bf Main Results.}
Our main results are presented in  {\bf  \Cref{thm-A}} and   {\bf \Cref{thm:homogenization}}. 
The former is an enhanced and rough functional limit theorem for multi-dimensional Voterra processes: $X^\epsilon\Rightarrow X$ and under more stringent decay conditions on the Hermite coefficients, we also show $(X^\epsilon, \XX^\epsilon)\Rightarrow (X, \XX)$ where $X$ is a Gaussian field.
\Cref{thm:homogenization} is an application to the homogenization problem for \eqref{eq:odenoisy}.
We show that, for $f\in \C_b^3$, the limiting effective equation is an SDE of Kunita type:   
$$
dx_t=\Gamma(x_t) dt+F(x_t, dt),
$$ 
where $F(x,\cdot)\define X(x)$ is a martingale with spatial parameters whose  characteristics are of the form $a_{i,j}(x,z,t)=\sigma_{i,j}(x,z)t$. The advantage of the formulation as a Kunita SDE is that for any finite set of initial positions we have the weak convergence of the $N$-point motion, by which we mean $(\phi^\epsilon(z^1), \dots, \phi^\epsilon(z^N))$, where $\phi^\epsilon$ are the solution flows of the SDE and $z^i\in \R^d$ are initial points. This is due to the convergence of the drivers in the compactly supported case and solution theory of SDEs of Kunita type. These theorems are proved under further conditions on the Hermite rank of $f$.

The proof of  \Cref{thm-A} is the content of Section 3, which relies on computational techniques from Malliavin calculus, and that for \Cref{thm:homogenization} is presented in Section 4. 
For the proof of the convergence of $(x_t^\epsilon)_{t\in[0,T]}$, we take the continuity theorem route for solution flows of rough differential equations (RDE). For this 
we cast the equation \eqref{eq:odenoisy} as a Banach space-valued rough differential equation. 
To put this in the perspective of  diffusion processes, this  is analogous to lifting a smooth stochastic differential equation to the diffeomorphism group. The rough differential equation method has previously been employed in the very nice work of Kelly and Melbourne \cite{Kelly-Melbourne}, in a different context. We also resolve a question raised in \cite{Kelly-Melbourne} in \cref{prop:weak_convergence} below, with which we also bypass the invocation of the martingale problem method used there, for identifying the limit equation. We then  turn this problem back to a finite state problem, and interpret the resulting RDE as classical stochastic differential equations of Kunita type and identify the characteristics of the driving semi-martingales with spatial parameters.   Note that we have kept the infinite-dimensional noise in the Kunita type SDE,  for we would lose the simultaneous convergence if it is converted to an SDE driven by a finite-dimensional noise. 

\paragraph{Acknowledgements.} We acknowledge helpful comments from an anonymous referee.

\section{Preliminaries}\label{sec-preliminaries}

This section features a brief overview of background material and a preliminary treatment of the noise. 

\subsection{The Normalized Noise Process}\label{sec:volterra}
\begin{definition}\label{def:volterra}
	An $\R^n$-valued stochastic process $(y_t)_{t\geq 0}$ is called a \textit{(stationary) Volterra process} if there is an integer $d\in\N$, a square-integrable kernel $K:\R\to\Lin[d]{n}$ with $K(t)=0$ for $t<0$, and a $d$-dimensional Wiener process $(W_t)_{t\in\R}$ such that
	\begin{equation}\label{eq:kernel}
	(y_t)_{t\geq 0}\overset{d}{=}\left(\int_\R K(t-s)\,dW_s\right)_{t\geq 0}.
	\end{equation}
	Let $\beta,\Theta>0$. We write $\vol_n(\beta,\Theta)$ for the space of $n$-dimensional Volterra processes which satisfy
	\begin{equation}\label{eq:decay_estimate_tail}
	\left|\int_{-\infty}^{-t}\Braket{K_j(-u), K_j(-u)}\,du\right|\leq \Theta \big(1 \wedge t^{- 2 \beta}\big)\qquad\forall\,t\geq 0,\,j=1,\dots,n,
	\end{equation}
	where $K_j$ is the $j^{\text{th}}$ row of $K$. It is also convenient to declare $\vol_n\define\bigcup_{\beta,\Theta>0}\vol_n(\beta,\Theta)$.
\end{definition}

It is clear that any $y\in\vol_n$ is a centered, stationary Gaussian process, which is actually ergodic under the canonical time-shift \cite{Cornfeld-Fomin-Sinai}. We set $\Sigma\define\Expec{y_0\otimes y_0}$ throughout the article.

By a simple application of the Cauchy-Schwarz inequality, the estimate \eqref{eq:decay_estimate_tail} yields
\begin{equation*}
	\left|\int_\R\Braket{K_i(s-u), K_j(t-u)}\,du\right|\leq \Theta \left( 1 \wedge |t-s|^{- \beta} \right)\qquad\forall\,s,t\geq 0,\,i,j=1,\dots,n.
\end{equation*}
This in turn implies the following decay of the temporal correlations of $(y_t)_{t\geq 0}$
\begin{equation*}
  \big|\Expec{y_s\otimes y_t}\big|\leq \Theta \left( 1\wedge|t-s|^{-\beta}\right).
\end{equation*}
 Let us give a few examples comprised by \cref{def:volterra}:
\begin{example}
  \leavevmode
  \begin{enumerate}
	\item Given a fractional Brownian motion $B$ with Hurst parameter $H \in (0,1)\setminus\{\frac12\}$, one can show that for functions $f$ satisfying 
   \begin{equation*}
   \begin{cases}\int_{\R} \left( \int_{\R} f(u) (u-s)_{+}^{H- \f 3 2} \,du \right)^2\, ds<\infty & H>\frac12,\\[10pt]
   \exists \phi_f \in L^2(\R): f(s) = \int_{\R} \phi_f(u) (u-s)_{+}^{-H- \f 1 2}\,du & H<\frac12,
   \end{cases}
   \end{equation*}
   the following identity holds
  \begin{equation*}
	 \int_{\R} f(s) dB_s = c_H\begin{cases} \displaystyle\int_{\R} \int_{\R}f(u) (u-s)_{+}^{H- \f 3 2} \,du\,dW_s, & H>\frac12,\\[10pt]
	  \displaystyle\int_{\R}  \int_0^{\infty} \f{ f(s) - f(s+u)} {u^{\f 3 2 - H}}\,du \,dW_s, & H<\frac12,
  \end{cases}
  \end{equation*} 
  see \cite{Pipiras-Taqqu}. Here, $c_H>0$ is some explicitly known constant. Hence, the fractional Ornstein-Uhlenbeck process $dX_t=-X_t\,dt+dB_t$ \cite{Cheridito2003} is a Volterra process with kernel
  \begin{equation*}
	K(t) = c_H\begin{cases}
	\displaystyle\int_{0}^\infty e^{-u} (u-t)_{+}^{H- \f 3 2}\,du, & H>\frac12,\\[10pt]
	\displaystyle e^{-t} \int_{0}^{\infty}  \f {\1_{[0,\infty]}(t) -  e^{-u} \1_{[-\infty,t]}(u)} {u^{\f 3 2 -H} }\,du, & H<\frac12.
	\end{cases}
  \end{equation*}
  It is not hard to check that \eqref{eq:decay_estimate_tail} holds with $\beta=1-H$.
  \item Another example is
  $$ K(t) =  c_H\begin{cases}
	  0 &\text{ if } t<0,\\
	  t^{H- \f 1 2} &\text{ if } 0 \leq t \leq 1,\\
	  t^{H- \f 1 2} - (t-1)^{H-\f 1 2} &\text{ if } t \geq 1,
  \end{cases}\qquad H\in(0,1),
  $$
  which leads (for a suitable choice of the normalization constant $c_H\in\R$) to a fractional Brownian increment $\int_{\R} K(t-s)\,dW_s = B_t - B_{t-1}$ \cite{Mandelbrot1968}. It is again easy to see that \eqref{eq:decay_estimate_tail} holds with $\beta = 1-H$.
  \end{enumerate}
\end{example}

\begin{fact}\label{fact:diagonalization}
  Let $\Sigma\in\Lin{n}$ be a positive semi-definite matrix with rank $m$. Then there is an isometry $O\in\Lin[m]{n}$ such that $D\define O^{\top} \Sigma O$ is diagonal and features precisely the non-zero eigenvalues of $\Sigma$. Let $G$ be a real-valued function on $\R^n$. Then 
  \begin{equation*}
	G\in L^2\big(\R^n,N(0,\Sigma)\big)\quad\Longleftrightarrow\quad G\big(OD^{\frac12}\cdot\big)\in L^2\big(\R^m,N(0,\id)\big).
  \end{equation*}
\end{fact} 

Let $(y_t)_{t\geq 0}$ be a centered, stationary Gaussian process and recall that $\Sigma=\Expec{y_0\otimes y_0}$. There is no loss of generality in assuming that $\rank(\Sigma)=n$ for otherwise $y$ lives almost surely in a proper subspace of $\R^n$. Without further notice, we shall resort to this case in the sequel. Let $D$ and $O$ denote the matrices furnished by \cref{fact:diagonalization}. For any $G\in L^2\big(\R^n,N(0,\Sigma)\big)$, we have the following $L^2\big(\R^n,N(0,\id)\big)$ convergent expansion:
\begin{equation*}
  G\big(OD^{\frac12}y\big) = \sum_{\bl\in \N^n_0} c_{\bl} H_{\bl}(y),\qquad\text{where}\quad c_{\bl}\define\int_{\R^n} G\big(OD^{\frac12}y\big) H_{\bl}(y)\,N(0,\id)(dy).
\end{equation*}
Here, $H_{\bl}:\R^n\to\R$ denotes the \textit{Hermite polynomial} of degree $\bl$:
\begin{equation*}
  H_{\bl}(x)\define\prod_{i=1}^n H_{\bl_i}(x_i),\qquad x=(x_1,\dots,x_n)\in\R^n,
\end{equation*} 
where
\begin{equation*}
  H_m(x) \define (-1)^m e^{\f {x^2} {2}} \f {d^m} {dx^m} e^{-\f {x^2} {2}},\qquad x\in\R,\,m\in\N_0.
\end{equation*}
Note that $H_0(x)=1$, $H_1(x)=x$ and $\< H_m , H_n \>_{L^2(\R,N(0,1))} =  \delta_{m,n} m!$. We want to remark that in \cite{Nualart} the Hermite polynomials are defined with a different normalization.

 Let us introduce the \textit{normalized process} $(z_t)_{t\geq 0}$ by declaring 
 \begin{equation}\label{eq:normalized_process}
   z_t\define D^{-\frac12}O^\top y_t.
 \end{equation} 
 Then $(z_t)_{t\geq 0}$ is clearly a centered, stationary Gaussian process with $\Expec{z_0\otimes z_0}=\id$ and 
\begin{equation}\label{eq:hermite_expansion}
  G(y_t)=G\big(O D^{\f 1 2} z_t\big) = \sum_{\bl\in \N^{n}_0} c_{\bl} H_{\bl}(z_t),\qquad t\geq 0.
\end{equation}
Note that, by definition, 
\begin{equation}\label{eq:transformed_kernel}
  z_t = \int_{\R} \hat{K}(t-u) dW_u,\qquad \hat{K}(t)\define D^{-\f 1 2} O^{\top}K(t),
\end{equation}
and $z\in\vol_n(\beta,\hat\Theta)$ where
$\hat{\Theta}=\big|D^{-\frac12}O^{\top}\big|^2\Theta$.

\begin{definition}\label{def-Hermite-rank} 
	Let $G\in L^2\big(\R^n,N(0,\Sigma)\big)$ and consider the expansion \eqref{eq:hermite_expansion}. 
	\begin{enumerate}
		\item The \textit{Hermite rank} of $G$ (with respect to $y$) is defined by
		\begin{equation*}
		\herm(G)\define\inf\big\{|\bl|:\,c_{\bl}\neq 0\big\}.
		\end{equation*}
		\item\label{it:fast_decay} We say that $G$ satisfies the \textit{fast chaos decay condition} with parameter $p>1$ if
		\begin{equation*}
		\sum_{\bl\in\N_0^n}|c_{\bl}| (p-1)^{\frac{|\bl|}{2}} \sqrt{\bl!} < \infty.
		\end{equation*}
	\end{enumerate}
\end{definition}
\begin{remark}
Denoting the  Ornstein-Uhlenbeck operator by
$$ T_\theta G = \sum_{\bl \in \N_0^n} c_{\bl} e^{-\theta |\bl|} H_{\bl},$$
we have  $\D(T_{\theta}) = \{ G \in L^2(\R^n,N(0,\id)) : \sum_{\bl\in\N_0^n}|c_{\bl}|^2 e^{-2\theta |\bl|} \bl!<\infty \}. $ Thus, by an application of Cauchy-Schwarz we obtain that a function $G$ satisfies the fast chaos decay condition with parameter $p$ if for some $\delta >0$, $G  \in \D(T_{-\f 1 2 \ln(p-1) - \delta})$.
\end{remark}
\begin{example}
		  The generating functions
	\begin{equation*}
	e^{\braket{x,a}-\frac{|a|^2}{2}}=\sum_{\bl\in\N_0^n}\frac{a^{\bl}}{\bl!}H_{\bl}(x),\qquad a,x\in\R^n,
	\end{equation*}
satisfy the fast chaos decay condition.
\end{example}

\subsection{Malliavin Calculus}\label{sec:malliavin}

In this section we recall some concepts from Malliavin calculus. For details we refer to \cite{Nualart}.

Let $ \H=L^2(\R, \R^d)$, w.r.t. the Lebesgue measure. If $W_t=(W^1_t, \dots, W^d_t)$ is a two-sided Wiener process, we construct an {\it isonormal Gaussian process}  $\{W(h): h \in \H\}$ by It\^o-Wiener integrals
$$W(h)\define \int_\R \<h(t), dW_t\>,$$
where $h$ is identified with an $L^2$ function from $\R$ to $\R^d$. If $f \in \H^{\otimes m}\cong L^2(\R^m, \R^{d^m})$,  it corresponds to a family of functions  $f_{i_1, \dots, i_m}: \R^m\to \R$ where $i_j \in \{1, \dots, d\}$. We denote by $f^\Sym$ its symmetrization.
We  define the multiple Wiener integral with respect to $W$ as follows:
$$\begin{aligned} I \colon& \bigoplus_{m \ge 0} \H^{\otimes m} \to L^2(\Omega; \R), \qquad \qquad \hbox{for } f\in  \H^{\otimes m},\\
I(f) \define  & m! \sum_{i_1, \dots, i_m =1}^d \int_{-\infty}^{\infty}\! \int_{-\infty}^{t_{m-1}}\! \dots \!\int_{-\infty}^{t_2} 
f^\Sym_{i_1, \dots, i_m} (t_1, \dots, t_m) dW^{i_1}_{t_1}\cdots dW^{i_m}_{t_m}.
\end{aligned}$$
Given $f=(f_1,\dots,f_n)\in\H^n$ and a multi-index $\bl=(\bl_1, \dots, \bl_n)\in\N_0^n$, we define  $f^{\otimes\bl}\in\H^{\otimes|\bl|}$ by
\begin{equation}
f^{\otimes\bl}\define f_1^{\otimes\bl_1}\otimes\cdots\otimes f_n^{\otimes\bl_n}.
\end{equation} 
To keep the verification of the fourth moment theorem simple we introduce contraction
operators between functions that are not necessarily symmetric.
Let $S$ and $\tilde S$ be two index sets, $f=\otimes_{i\in S} f_i$ and $g=\otimes_{i\in \tilde S} g_i$ are primitive tensors with $f_i, g_i\in \H$. For any pair $p=\{a,b\}$ where 
$a\in S$ and $b\in \tilde S$ we define the contraction of $f\otimes g$:
\begin{equation}
\Tr_p(f\otimes g)=\<f_a, g_b\> \; \bigg( \bigotimes_{i\in S\setminus\{a\}} f_i \bigg) \otimes\bigg( \bigotimes_{i\in \tilde S\setminus\{b\}} g_i\bigg).
\end{equation}
We define the multi-contraction with multi-pairs of $k$-elements. Let $P=\{(a_i,b_i)\}_{i=1}^k$ be a set of complete pairings, where $\{a_i\}_{i=1}^k\subset S$ and $\{b_i\}_{i=1}^k\subset\bar S$. We set
\begin{equation}
\Tr_P(f\otimes g)=\prod_{i=1}^k \<f_{a_i}, g_{b_i}\> \; \bigg( \bigotimes_{i\in S\setminus\{a_1,\dots,a_k\}} f_i \bigg) \otimes\bigg( \bigotimes_{i\in \tilde S\setminus\{b_1,\dots,b_k\}} g_i\bigg).
\end{equation}
Then the following product formula holds for $f\in \H^{\otimes m} $ and $g \in \H^{\otimes n} $:
\begin{equation}\label{product-formula}
I(f)I(g) =\delta_{n,m} \sum_{P\in {\mathcal P} } I( \Tr_P  f\otimes g),      \qquad {\mathcal P}=\bigcup_{k=0}^n { \mathcal P}_k,
\end{equation}
where $P$ runs through all multi-pairs from  $S=\{1, \dots, m\}$ and $\tilde S=\{1, \dots, n\}$ so ${\mathcal P}=\bigcup_{k=0}^{m\wedge n}{\mathcal P}_k$, where   ${\mathcal P}_k$ denotes the collection of all 
$k$ distinct pairs of indices from $S\times \tilde S$. The $0^{\textup{th}}$ contraction is $f\otimes g$.  In particular,
\begin{equation}\label{eq:isometry}
\E \left[ I(f) I(g) \right] =  \delta_{m,n} m! \Braket{\Sym(f), \Sym(g)}_{\H^{\otimes m}}.
\end{equation}

We have the following straight-forward generalization of \cite[Proposition 1.1.4]{Nualart}:
\begin{lemma}\label{lem-Hermite-to-st-integral}
  Let $\bl\in\N_0^n$ and $f \in\H^n$ with $\|f_i\|_{\H} =1$ for each $i=1,\dots,n$. Then we have
  \begin{equation*}
	H_{\bl}\big(W(f)\big)=I(f).  
  \end{equation*} 
\end{lemma}

Given $ \bl \in \N^n_0$ we call a graph of complete pairings,  without self-loops and with nodes $\{1, \dots, n\}$ \textit{$\bl$-admissible} if the $k^{\text{th}}$ node has exactly $\bl_k$ edges. We denote the collection of all such graphs by $\Gamma_{\bl}$. We will use the fact
\begin{equation}\label{eq:number_graphs}
   \vert \Gamma_{\bl} \vert \leq \sqrt{\bl!} (n-1)^{\f {\vert \bl \vert}{2}}.
 \end{equation} 
For a graph $\G\in\Gamma_{\bl}$, we write $\gamma_{i,j}(\G)$ for the number of edges between the nodes $i$ and $j$.

\begin{proposition}[Diagram Formula \cite{BenHariz, Graphsnumber}]\label{thm-diagramm-formulae}
Let $\bl\in\N_0^n$ and $X=(X_1,\dots,X_n)$ be multivariate Gaussian. Then
$$ \Expec{H_{\bl}(X)}= \sum_{\G\in \Gamma_{\bl}} \prod_{i=1}^n \prod_{j=1}^i \Expec{X_i X_j}^{\gamma_{i,j}(\G)}.$$
In particular, if $\bk \in \N^n_0$, $Y=\left( Y_1, \dots, Y_n \right)$ is multivariate Gaussian jointly with $X$, and both $X$ and $Y$ have pairwise independent components, then we have 
\begin{equation*}
  \big|\Expec{H_{\bl}(X) H_{\bk}(Y)}\big| \lesssim \delta_{| \bk |, | \bl |} \sqrt{\bk! \bl!}(2n-1)^{|\bl| }   \max_{1 \leq i\leq j \leq n} \big|\Expec{X_i Y_j}\big|^{|\bl|}.
\end{equation*}
\end{proposition}

\subsection{H\"older Spaces and Their Tensor Product}\label{sec:tensor_product}

The algebraic tensor product $\CX\otimes_a\CY$ of two vector spaces $\CX$ and $\CY$ is defined as the subspace of the dual of the bilinear mappings $\CX\times\CY\to\R$ spanned by the elements $x\otimes y$, $x\in\CX$, $y\in\CY$, which act by
\begin{equation*}
  \braket{B,x\otimes y}\define B(x,y).
\end{equation*}
We can also declare a dual action on $\CX\otimes_a\CY$ by $\braket{x\otimes y,x^*\otimes y^*}\define \braket{x,x^*}\braket{y,y^*}$ for $x^*,y^*$ in the dual spaces of $\CX$ and $\CY$, respectively.

If $\CX$ and $\CY$ are Banach spaces, we call a norm $\|\cdot\|_{\CX\otimes\CY}$ on $\CX\otimes_a\CY$ a \textit{reasonable crossnorm} if
\begin{alignat*}{4}
  \|x\otimes y\|_{\CX\otimes\CY}&=\|x\|_{\CX}\|y\|_{\CY}&\qquad&\forall\,x\in\CX,\,y\in\CY,\\
  \|x^*\otimes y^*\|_{(\CX\otimes\CY)^*}&\define\sup_{z\in\CX\otimes_a\CY}\frac{|\braket{z,x^*\otimes y^*}|}{\|z\|_{\CX\otimes\CY}}=\|x^*\|_{\CX^*}\|y^*\|_{\CY^*}&\qquad&\forall\,x^*\in\CX^*,\,y^*\in\CY^*.
\end{alignat*}
It is easy to see that both canonical examples, the injective and the projective tensor norm, define reasonable cross norms. Finally, the tensor product of $\CX$ and $\CY$ is defined as the completion of $\CX\otimes_a\CY$ with respect to the norm $\|\cdot\|_{\CX\otimes\CY}$. Without further notice, we always assume that $\CX\otimes\CY$ is equipped with a reasonable crossnorm. Further details on the tensor product of Banach spaces can be found in the classical monographs \cite{Light1985,Ryan2002}.

We now turn to the tensor product of interest in the sequel of this work. Let $\alpha>0$ and $\CX$ be a normed space. The classical H\"older space $\C^{\alpha}_b(\R^d,\CX)$ is defined as the collection
\begin{align}
  \C^{\alpha}_b(\R^d,\CX)&\define\big\{f\in\C^{\lfloor\alpha\rfloor}_{b}(\R^d,\CX):\,|f|_{\C^{\alpha}_b}<\infty\big\},\nonumber
  \end{align}
  with the  norm
  \begin{align}
  |f|_{\C^{\alpha}_b}\define \sup_{\substack{\bl\in\N_0^d\\|\bl|\leq\lfloor\alpha\rfloor}}&\sup_{x\in\R^d}|D^{\bl}f(x)|_{\CX}
  + \sup_{\substack{\bl\in\N_0^d\\|\bl|=\lfloor\alpha\rfloor}}
  \sup_{x\neq y}\frac{|D^{\bl}f(x)-D^{\bl}f(y)|_\CX}{|x-y|^{\alpha-\lfloor\alpha\rfloor}}. \label{eq:holder_norm}
\end{align}
We also write $f\in\C^{\alpha+}_b(\R^d,\CX)$ if there is an $\alpha^\prime>\alpha$ such that $f\in\C^{\alpha^\prime}_b(\R^d,\CX)$.

We will use the following result. It was shown in \cite[Cor. 4.6]{Kelly-Melbourne} for a norm equivalent to \eqref{eq:holder_norm}. 
\begin{lemma} \label{lem:tensor_product}
The canonical embedding $$
\imath: \C_b^{\alpha}(\R^d,\R^d)\otimes\C_b^{\alpha}(\R^d,\R^d) \hookrightarrow \C_b^{\alpha}\big(\R^d,\C_b^{\alpha}(\R^d,\R^d\otimes\R^d)\big),$$
 \begin{equation*}
	\imath(f\otimes g)(x,y)\define f(x)\otimes g(y)\qquad\forall\,f,g\in\C_b^{\alpha}(\R^d,\R^d),\,x,y\in\R^d,
  \end{equation*}
extending to the whole space  by linearity,   defines a reasonable cross norm  on the former, with respect to which the embedding is continuous.
\end{lemma}

\subsection{Rough Path Theory}\label{sec:rough_path}

The theory of rough paths has by now certainly found its way into the mathematical mainstream. This is, of course, also due to the very nice monographs \cite{Friz-Victoir,Friz-Hairer} to which we refer for further details. We shall work in the framework of controlled rough paths \cite{Gubinelli-lemma,Feyel-delaPradelle} popularized in the book of Friz and Hairer. 

Let $X,Y:[0,T]\to\CX$ be H\"older continuous with exponents $\gamma_1$ and $\gamma_2$, respectively. If $\gamma_1 + \gamma_2 >1$, Young's integration theory enables us to define $\int_0^T Y dX$ as the limit of Riemann sums $\sum_{[u,v]\in \mathcal P} Y_u(X_v-X_u)$ along any sequence of partitions $\CP$ of $[0,T]$ with mesh $|\CP|$ tending to $0$ \cite{Young1936}. Furthermore, the mapping $(X,Y)\mapsto \int_0^\cdot Y dX$ is continuous. Thus, for $X\in \C^{\f 12+}$, one expects a unique solution theory to the Young differential equation $dY_s=f(Y_s) dX_s$ (given enough regularity on $f$). Indeed, if $f\in \C_b^2$, this equation is well posed and the solution is  continuous in both the driver $X$ and the initial data.
In the case $\gamma_1\leq\f 12$, this fails and one cannot define the integral $\int X dX$ by the above Riemann sum anymore. This is partially remedied by rough path theory, which allows to define an integral with respect to less regular integrators by enhancing the Riemann sum, see \eqref{eq:enhanced} below. 

Let $T>0$, $\gamma\in(0,1)$, and $\CX$ be a Banach space. We write $|X|_{\C^\gamma}$ for the $\gamma$-H\"older norm of a function $X:[0,T]\to\CX$. The space of all $\gamma$-H\"older functions $[0,T]\to\CX$ is denoted by $\C^\gamma\big([0,T],\CX\big)$. We set $\Delta_T\define\{(s,t):\,0\leq s\leq t\leq T\}$ and
\begin{equation*}
  \C^{\gamma}(\Delta_T,\CX\otimes\CX)\define\big\{\XX:\Delta_T\to\CX\otimes\CX:\,|\XX|_{\gamma}<\infty\big\},\qquad |\XX|_{\gamma}\define\sup_{(s,t)\in\Delta_T}\frac{\|\XX_{s,t}\|_{\CX\otimes\CX}}{|t-s|^\gamma},
\end{equation*}
where $\|\cdot\|_{\CX\otimes\CX}$ is the norm on $\CX\otimes\CX$. Note that any $\gamma$-H\"older function $X\in\C^\gamma\big([0,T],\CX\big)$ has a natural lift to $\C^{\gamma}(\Delta_T,\CX)$ by declaring $X_{s,t}\define X_t-X_s$. A $\gamma$-rough path with values in $\CX$ is a pair $\X=(X,\XX)\in\C^{\gamma}\big([0,T],\CX\big)\oplus\C^{2\gamma}\big(\Delta_T,\CX\otimes\CX\big)$ obeying the algebraic constraint (\textit{Chen's relation})
\begin{equation}\label{eq:chen}
  \XX_{s,t}-\XX_{s,u}-\XX_{u,t}=X_{s,u}\otimes X_{u,t}\qquad\forall\, 0\leq s\leq u\leq t\leq T.
\end{equation}
Owing to this identity, the values of the two-parameter process $\XX$ can actually be recovered from the knowledge of $\XX_t\define\XX_{0,t}$, $t\in[0,T]$:
\begin{equation*}
  \XX_{s,t}=\XX_{t}-\XX_{s}-X_s\otimes X_t.
\end{equation*}
It is customary to denote the set of $\gamma$-rough paths with values in $\CX$ by $\FC^\gamma\big([0,T],\CX\big)$. Albeit this space is certainly not linear, it becomes a complete metric space in the topology inherited from the Banach space $\C^{\gamma}\big([0,T],\CX\big)\oplus \C^{2\gamma}\big([0,T],\CX^2\big)$ which is canonically equipped with the norm
\begin{equation*}
  |\X|_{\C^{\gamma}\oplus\C^{2\gamma}}\define|X|_{\C^\gamma}+|\XX|_{\C^{2\gamma}}.
\end{equation*}
We emphasize that---unless $\CX=\{0\}$---the space $\FC^\gamma\big([0,T],\CX\big)$ is not separable. In order to avoid norm versus seminorm considerations, we shall tacitly assume that $X_0=0$ which is anyways the case in our ultimate application of the theory.

Let $f\in\C^2\big(\R^d,\L(\CX,\R^d)\big)$, $Y\in\C^\gamma([0,T],\R^d)$, and $\X\in\FC^\gamma([0,T],\CX)$. For $\gamma\in(\frac13,\frac12]$ the integral
\begin{equation}\label{eq:rough_integral}
   \int_0^\cdot f(Y_s)\,d\X_s\in\C^\gamma([0,T],\R^d)
\end{equation} 
can be defined, provided that $Y$ is controlled by $X$. This is to say, there is a $Y^\prime\in\C^\gamma\big([0,T],\L(\CX,\R^d)\big)$ such that
\begin{equation*}
  \big|Y_{s,t}-Y_s^\prime X_{s,t}\big|=\CO\big(|t-s|^{2\gamma}\big),\qquad 0\leq s,t\leq T.
\end{equation*}
In this case, it is also customary to write $(Y,Y^\prime)\in\FD_X^{2\gamma}([0,T],\R^d)$. Under these conditions, the integral \eqref{eq:rough_integral} is well defined as the limit of \textit{compensated} Riemann sums along an arbitrary sequence of partitions with mesh tending to $0$:
\begin{equation}\label{eq:enhanced}
  \int_s^t f(Y_s)\,d\X_s\define\lim_{|\CP|\to 0}\sum_{[u,v]\in\CP([s,t])}\Xi_{u,v},\qquad\Xi_{u,v}\define\Big( f(Y_u) X_{u,v}+\big(Y_u^{\prime}\odot Df(Y_u)\big)\XX_{u,v}\Big).
\end{equation}
Here, we introduced the `product'
\begin{equation*}
  \big(Y_u^{\prime}\odot Df(Y_u)\big)(x\otimes y)\define\sum_{i=1}^d \big(Y_u^\prime(x)\big)_i \partial_i f(Y_u)(y),
   \qquad x,y\in\CX.
\end{equation*}
Employing the algebraic relation \eqref{eq:chen}, it is then an easy exercise to check that
\begin{equation*}
  \delta\Xi_{s,u,t}\define\Xi_{s,t}-\Xi_{s,u}-\Xi_{u,t}=\CO\big(|t-s|^{3\gamma}\big),\qquad 0\leq s\leq u\leq t\leq T,
\end{equation*}
whence the integral \eqref{eq:rough_integral} is indeed well defined by the sewing lemma \cite{Gubinelli-lemma,Feyel-delaPradelle}.

\begin{proposition}\label{prop:rde_existence}
  Let $Y_0\in\R^d$ and $\gamma\in\big(\frac13,\frac12\big)$. If $\X\in\FC^{\gamma}([0,T],\CX)$ and $f\in\C_{b}^{\f 1 \gamma+}\big(\R^d,\L(\CX,\R^d)\big)$, then there is a controlled rough path $\big(Y,f(Y)\big)\in\D_X^{2\gamma}([0,T],\R^d)$ such that
  \begin{equation*}
	Y_t=Y_0+\int_0^t f(Y_s)\,d\X_s,\qquad t\in[0,T].
  \end{equation*}
  Moreover, the solution map $\Phi(Y_0,\X)\define Y$ is continuous from $\R^d\times\FC^{\gamma}([0,T],\CX)$ to $\C^{\gamma^\prime}([0,T],\R^d)$ for each $\gamma^\prime<\gamma$.
\end{proposition}
\begin{remark}
  We note that the theory of controlled rough paths usually assumes the continuous embedding
  \begin{equation}\label{eq:rp_embedding}
	\L\big(\CX,\L(\CX,\R^d)\big)\hookrightarrow\L(\CX\otimes\CX,\R^d),
  \end{equation}
  see \cite[Section 1.5]{Friz-Hairer}. It is well known that this is ensured by working with the projective tensor product on $\CX\otimes\CX$. This, however, turns out to be rather inconvenient when $\CX$ is a H\"older space, which is the case of interest in this article. It is in fact beneficial to work with the more explicit reasonable crossnorm introduced in \cref{sec:tensor_product}. Thankfully, as already observed by Kelly and Melbourne \cite[Proof of Theorem 3.3]{Kelly-Melbourne}, the proof of \cref{prop:rde_existence} in \cite{Friz-Hairer} does not really require the embedding \eqref{eq:rp_embedding}, but it is enough to equip the tensor product with a reasonable crossnorm.
\end{remark}

\section{The Multidimensional Limit Theorem}

\subsection{Statement of the Result}

First we need to introduce a bit of notation: $\WW$ denotes the Stratonovich lift of an $n$-dimensional standard Wiener process $W$, that is, $\WW^{i,j}_{s,t}=\int_s^t \big(W_r^i-W_s^i)\,\circ dW_r^j$. 
For a smooth function $X$, we let $\XX_{s,t}\define\int_s^t (X_{r}-X_s)\,\otimes dX_r$ denote the canonical lift and set $ (\Upsilon\otimes \Upsilon) \XX_{s,t} \; \define \; \int_s^t \Upsilon X_{s,r} d (\Upsilon X)_r$ where $\Upsilon$ is a matrix.

\begin{theorem}\label{thm-A}
	Let $y\in\vol_n(\beta,\Theta)$ for some $\beta,\Theta>0$ and 
	$\hat{\Theta}=\big|D^{-\frac12}O^{\top}\big|^2\Theta$.  Then for any real-valued functions $G_1,\dots,G_N\in L^2\big(\R^n,N(0,\Sigma)\big)$ with
	\begin{equation*}
	\min_{k=1,\dots,N}\herm(G_k)>\frac1\beta
	\end{equation*}
	and any terminal time $T>0$ all of the following hold: 
	\begin{itemize}
		\item \textbf{\textup{(Finite dimensional distributions)}} If each $G_k$ is of fast chaos decay with parameter $\hat{\Theta} (2n-1) +1$, then
		\begin{equation}\label{eq:main_fdd_convergence}
		X^{\epsilon}_t=  \left( \sqrt{\epsilon} \int_0^{\frac{t}{\epsilon}} G_1(y_s) \,ds , \dots, \sqrt{\epsilon} \int_0^{\frac{t}{\epsilon}} G_N(y_s) \,ds \right)_{t\in[0,T]}\fdd \Upsilon W,
		\end{equation}
		where $W=(W_t)_{t\in[0,T]}$ is a standard Wiener process and $\Upsilon$ is the unique non-negative square root of the matrix
		\begin{equation*}
		\Upsilon^2_{i,j}\define \int_0^{\infty}\Expec{G_{i}(y_r)  G_{j}(y_0) +  G_{i}(y_0)  G_{j}(y_r)} \,dr,\qquad i,j=1,\dots,N.
		\end{equation*}
		\item \textbf{\textup{(Functional Central Limit Theorem)}} If each $G_k$ is of fast chaos decay with parameter $ \hat{\Theta}  (2n-1)(p_k-1)+1$ for $p_k>2$,  then the convergence \eqref{eq:main_fdd_convergence} is actually weakly in $\C^{\gamma}\big([0,T], \R^N\big)$ for any $\gamma<\frac12-(\min_k p_k)^{-1}$.
		
		\item \textbf{\textup{(Rough Central Limit Theorem)}} If $\min_{k=1,\dots,N}\herm(G_k)>2\beta^{-1}$ and each $G_k$ is of fast chaos decay with parameter $\hat{\Theta} (4n-1)(p_k-1)+1$, then, for any $\gamma <\f 1 2 - ( \min_k p_k )^{-1}$, where $p_k > 2$, 
		$$ \X^{\epsilon}=\big(X^\varepsilon,\XX^\varepsilon\big) \Rightarrow \X$$
		in the rough path space $\FC^{\gamma}\big([0,T],\R^N\big)$, where 
		\begin{equation}\label{eq:area}\begin{aligned}
		\X_{s,t}&= \big( \Upsilon W_{s,t}, \quad  ( \Upsilon \otimes \Upsilon ) \WW_{s,t}+\Xi(t-s)\big),\\
		\Xi_{i,j} &= \f 1  2 \int_0^{\infty} \E\left[ G_i(y_0) G_j(y_r) -  G_i(y_r) G_j(y_0) \right]\, dr.
		\end{aligned}
		\end{equation}
	\end{itemize}
\end{theorem}
\Cref{thm-A} is a multi-dimensional version of the result in \cite{Gehringer-Li-2020-1}, see also \cite{Gehringer2020} for related limit theorems for Hermite processes.

\begin{remark}\label{rem:stratonovich_correction}
	The limiting rough path \eqref{eq:area} can be rewritten in It\^o form as follows. Let $\WW^{\text{It\^o}}$ be the It\^o lift of $W$. Then
	\begin{equation*}
		\X_{s,t}=\big(\Upsilon W_{s,t}, \quad (\Upsilon\otimes\Upsilon)\WW_{s,t}^{\text{It\^o}} + \Lambda(t-s)\big),
	\end{equation*}
	where
	\begin{equation*}
		\Lambda_{i,j}=\int_0^\infty\Expec{G_i(y_0)G_j(y_r)}\,dr.
	\end{equation*}
\end{remark}

\subsection{Functional Central Limit Theorem}
We fix $y\in\vol_n(\beta,\Theta)$ and write $K$ for its kernel. Let $z$ be the normalized process \eqref{eq:normalized_process}. Recall, $z\in\vol_n(\beta,\hat\Theta)$ where $\hat{\Theta}=\big|D^{-\frac12}O^{\top}\big|^2\Theta$. We also set $\tau_t\hat K^i(s)\define\hat K^i(t-s)$, $s\in\R$, for the normalized kernel defined in \eqref{eq:transformed_kernel}.

\begin{lemma}\label{lem-decomposition}
For each $\bl\in\N_0^n$, $H_\bl(z_t)$ can be written in terms of the multiple It\^o-Wiener integral as follows:
\begin{equation}
H_{\bl}(z_t)= \prod _{i=1}^n  I\big( (\tau_t\hat{K}^i)^{\otimes \bl_i} \big)= I\big((\tau_t\hat{K})^{\otimes \bl} \big).
\end{equation}
Consequently, $\{H_\bl(z_t), \bl\in\N_0^n\}$ is an orthonormal set for any $t\geq 0$.
\end{lemma}
\begin{proof}
Given  a multi-index $\bl=(\bl_1, \dots, \bl_n)\in\N_0^n$, using \cref{lem-Hermite-to-st-integral},
$$H_{\bl}(z_t)= \prod_{i=1}^n  H_{\bl_i} ( W(\tau_t \hat K^i ))= \prod_{i=1}^n  I\big((\tau_t\hat{K}^i)^{\otimes \bl_i} \big).$$
Since  $\E[z_0\otimes z_0]=\id$, we see 
$$\< \tau_t \hat{K^i}, \tau_t\hat{K^j}\>=\int_\R \< \hat K^i(t-r), \hat K^j(t-r)\> dr
=\int_\R\< \hat K^i(-r), \hat K^j(-r)\>dr=\delta_{i,j},$$	
as claimed. 
\end{proof}


Let $G \in L^2\big(\R^n, N(0,\Sigma)\big)$.  The expansion \eqref{eq:hermite_expansion}
becomes
\begin{equation}\label{eq:normalized_hermite}
  G(y_t)= \sum_{\bl\in\N_0^n}c_{\bl}H_{\bl}(z_t)
  =
  \sum_{\bl\in\N_0^n}c_{\bl}I\big((\tau_t\hat{K})^{\otimes \bl}\big).
\end{equation}
The first step towards the proof of \cref{thm-A} is to  establish a central limit theorem for the finite-dimensional distributions of the vector-valued process 
\begin{equation}\label{eq:finite_dim}
  \left(  \sqrt{\epsilon} \int_0^{\f t \epsilon} G_1(y_s) ds, \dots, \sqrt{\epsilon} \int_0^{\f t \epsilon} G_N(y_s) ds  \right)_{t\in[0,T]}.
\end{equation}
The argument proceeds along a well-established pathway, see e.g. \cite{Bai-Taqqu, Gehringer2020,Nourdin-Nualart-Zintout}:
\begin{enumerate}
  \item\label{it:finite} We first assume $G_1,\dots,G_N\in L^2(\R^n,N(0,\Sigma))$ live in a finite number of chaoses (that is, their Hermite expansion \eqref{eq:normalized_hermite} is finite) and prove the statement of \cref{thm-A} by invoking the fourth moment theorem of Nualart, Peccati, and Tudor \cite{Nualart-Peccati,Peccati-Tudor}, see also \cite{Nualart2008,Nourdin2012}.
  \item A simple truncation argument then shows that the general case can be reduced to \ref{it:finite}.
\end{enumerate}

\begin{proposition}[\normalfont{Fourth Moment Theorem \cite[Theorem 5.2.7]{Nourdin2012}}]
	\label{fourth-moment-theorem}
	Let $m\geq 2$ and $\bl\in\N^m$ be a multi-index with ordered components, that is, $\bl_1\leq\cdots\bl_m$. 
	Let $(f^{\epsilon})_{\varepsilon>0} \subset \oplus_{i=1}^m \H^{\otimes \ell_i}$.  
	Denote by $f_i$ the projection of $f$ to $ \H^{\otimes \ell_i}$.  
	Assume that for $i,j=1,\dots,m$, the limit
	$$   \Lambda_{i,j}\define\lim_{\varepsilon\to 0}\Expec{I(f_i^\varepsilon)I(f_j^\varepsilon)}$$
	exists and 
	$$\lim_{\epsilon\to 0} \Tr_P f_k^\varepsilon\otimes f_k^\varepsilon=0 \qquad \forall  P\in \bigcup_{i=1}^{m-1}{\mathcal P}_i,$$
	i.e. all contractions of $f_k^\epsilon$ with itself vanish, except for the $0^{\textup{th}}$ and the $k^{\textup{th}}$. 
	Then  we have   $$\big(I_{\bl_1}(f^{\epsilon}_1), \dots, I_{\bl_m}(f^{\epsilon}_m) \big)\Rightarrow\CN(0,\Lambda).$$
\end{proposition}

By Lemma \ref{lem-decomposition}, we may apply the proposition to stochastic processes of the form  $\{  \sqrt{\varepsilon}\int_0^{\frac{t_i}{\varepsilon}} H_{\bl^j}\left(z_u \right)\,du\}$, which is done in the following two key lemmas.
\begin{lemma}\label{lem:limit}
	Let $\bk,\bl\in\N_0^n$  with $|\bl| \wedge |\bk| >\f 1 \beta$ and $s,t\in[0,T]$, then
	\begin{equation*}\begin{aligned}
	&   \lim_{\varepsilon\to 0}\varepsilon\int_{0}^{\frac{t}{\varepsilon}}\int_0^{\frac{s}{\varepsilon}}
	\Expec{ H_\bl (z_u) H_{\bk}(z_v)  }\,du\,dv
	=\delta_{\vert \bk \vert , \vert \bl \vert } (s\wedge t)
	 \int_0^\infty   \Expec{ H_\bl (z_u) H_{\bk}(z_0)  }  + \Expec{ H_\bl (z_0) H_{\bk}(z_u)  }\, du.
	\end{aligned}  
	\end{equation*}
\end{lemma}
\begin{proof}
	Using $H_\bl (z_u) =I\big(\tau_u\hat{K}^{\otimes \bk} \big)$ and by  \eqref{eq:isometry}, we may assume $|\bl|=|\bk|$ and also $s\leq t$. We have
	\begin{equation*}
	  \Expec{I\big(\tau_u\hat{K}^{\otimes \bk} \big)  I\big(\tau_v\hat{K}^{\otimes \bl} \big)}  = \Expec{
		\sum_{P\in {\mathcal P}} I \bigl( \Tr_P \big(\tau_u\hat{K}^{\otimes \bl} \big) \otimes \big(\tau_v\hat{K}^{\otimes \bk}) \bigr)  }.
	\end{equation*}  
	Let $\bar {\hat K}(t)\define  \tau_0 \hat K=\hat K(-t)$. 
	We will use the fact that for $v\le u$, the right-hand side is a sum of products of the form $\int_\R  \<\hat K^i(u-r), \overline{\hat K}^j(r-v) \>dr$, whence it is a function of $u-v$. We first take $s=t$, and by a change of variables,
	\begin{align*}
	\varepsilon\int_{0}^{\frac{s}{\varepsilon}}\int_0^{\frac{s}{\varepsilon}} \Expec{I\big(\tau_u\hat{K}^{\otimes \bl} \big)
		I\big(\tau_v\hat{K}^{\otimes \bk} \big)} \,du\,dv
	&= \int_{0}^{\frac{s}{\varepsilon}} \left( s-\varepsilon u \right) \Expec{I\big(\tau_u\hat{K}^{\otimes \bl} \big)   I\big(\tau_0\hat{K}^{\otimes \bk} \big)
		+I\big(\tau_0\hat{K}^{\otimes \bl} \big)   I\big(\tau_u{\hat{K}}^{\otimes \bk} \big)}\,du\\
	&\xrightarrow{\varepsilon\to0}  s
	 \int_0^\infty\Expec{I\big(\tau_u\hat{K}^{\otimes \bl} \big)   I\big(\tau_0\hat{K}^{\otimes \bk} \big)
		+I\big(\tau_0\hat{K}^{\otimes \bl} \big)   I\big(\tau_u{\hat{K}}^{\otimes \bk} \big)}\,du.
	\end{align*}
	Since
	\begin{equation*}
	\left|\varepsilon\int_{\frac{s}{\varepsilon}}^{\frac{t}{\varepsilon}}\int_{0}^{\frac{s}{\varepsilon}}
	\Expec{I\big(\tau_u\hat{K}^{\otimes \bl} \big)  I\big(\tau_v\hat{K}^{\otimes \bk} \big)} \,du\,dv\right|
	\lesssim\varepsilon\int_{\frac{s}{\varepsilon}}^{\frac{t}{\varepsilon}}\int_{0}^{\frac{s}{\varepsilon}} \big(1\wedge|v-u|^{-\beta|\bk|}\big)\,du\,dv\to 0,
	\end{equation*}
	the claim follows at once.
\end{proof}

Next, we show that the contractions vanish as we send $\varepsilon\to 0$. Note that, even though $\{\tau_t \hat{K^i}\}$ are orthonormal, for two different times $s\neq t$, $\tau_s \hat{K}^i$ and $\tau_t \hat{K}^j$ may not be orthogonal for $i\neq j$. If they were, then the proof of the next lemma would be much simpler.
\begin{lemma}\label{lem:fourth_moment_applies}
	Let $\bl\in\N_0^n$ and $t\in[0,T]$. If $|\bl| >\beta^{-1}$, then for $r=1, \dots, |\bl|-1$,
	and any multi-pair $P=\{(a_1, b_1), \dots, (a_r, b_r)\}$,  with both indices of a pair from $\{1, \dots, |\bl|\}$, one has
	$$  \varepsilon\Tr_P \int_{0}^{\frac{t}{\varepsilon}} \tau_u  K^\bl \,du\otimes \int_{0}^{\frac{t}{\varepsilon}} \tau_u  K^\bl  \,du\to 0.$$
\end{lemma}

\begin{proof}
	Let $\bk^1, \bk^2$ be the multi-indices obtained from $\bl$ by deleting $\{a_i\}$ and $\{b_i\}$ respectively. We compute
	\begin{align*}
	&\left\| \Tr_P \int_{0}^{\frac{t}{\varepsilon}} \tau_u  K^\bl \,du\otimes \int_{0}^{\frac{t}{\varepsilon}} \tau_u  K^\bl  \,du\right\| ^2  
	=\int_{\left[0,\frac{t}{\varepsilon}\right]^4} \< \Tr_P ( \tau_{u_1}  K^\bl \otimes  \tau_{u_2}  K^\bl ), \Tr_P ( \tau_{u_3}  K^\bl  \otimes  \tau_{u_4}  K^\bl )\>
	\,d^4u\\
	&\lesssim\int_{\left[0,\frac{t}{\varepsilon}\right]^4}\big(1\wedge|u_1-u_2|^{-\beta r}\big)\bigl(1\wedge|u_3-u_4|^{-\beta r }\bigr)
	|\<  \tau_{u_1}  K^{\bk^1} \otimes  \tau_{u_2}  K^{\bk^2}, \tau_{u_3}  K^{\bk^1} \otimes  \tau_{u_4}  K^{\bk^2}\>|
	\,d^4u\\
	&\lesssim\int_{\left[0,\frac{t}{\varepsilon}\right]^4}
	\big(1\wedge|u_1-u_2|^{-\beta r}\big)\big(1\wedge|u_3-u_4|^{-\beta r }\big)
	\big(1\wedge|u_1-u_3|^{-\beta(|\bl|- r )}\big)\big(1\wedge|u_2-u_4|^{-\beta(|\bl|- r)}\big)\,d^4u
	\end{align*}
	by the correlation decay of the kernel $K$. 
	Next, we use the  inequality $x^\alpha y^\beta\leq x^{\alpha+\beta}+y^{\alpha+\beta}$ for $x,y> 0$ and $\alpha,\beta\in\R$ to find 
	\begin{align*}
		&\left\| \Tr_P \int_{0}^{\frac{t}{\varepsilon}} \tau_u  K^\bl \,du\otimes \int_{0}^{\frac{t}{\varepsilon}} \tau_u  K^\bl  \,du\right\| ^2\\
		\lesssim&\int_{\left[0,\frac{t}{\varepsilon}\right]^4}
		\Big(\big(1\wedge|u_1-u_2|^{-\beta|\bl|}\big)+\big(1\wedge|u_1-u_3|^{-\beta|\bl|}\big)\Big)\big(1\wedge|u_3-u_4|^{-\beta r }\big)
		\big(1\wedge|u_2-u_4|^{-\beta(|\bl|- r)}\big)\,d^4u.
	\end{align*}
	We show that each of these summands is $o(\varepsilon^{-2})$ as $\varepsilon\to 0$. Since the arguments are similar, we shall only focus on the first term:
	\begin{align*}
		&\phantom{\lesssim}\int_{\left[0,\frac{t}{\varepsilon}\right]^4}\big(1\wedge|u_1-u_2|^{-\beta|\bl|}\big)\big(1\wedge|u_3-u_4|^{-\beta r}\big)\big(1\wedge|u_2-u_4|^{-\beta(|\bl|- r )}\big)\,d^4u\\
		&=\int_0^{\f t\epsilon} du_4\int_{-u_4}^{\f t\epsilon -u_4} \big(1\wedge|u_3|^{-\beta(|\bl|- r )}\big)du_3 \int_{-u_4}^{\f t\epsilon -u_4} \big(1\wedge|u_2|^{-\beta r}\big)du_2 \int_{-u_2+u_4}^{\f t\epsilon-u_2+u_4} \big(1\wedge|u_1|^{-\beta|\bl|}\big)du_1\\
		&\lesssim 
		\left(\f {t} \epsilon \right)^{2-\beta|\bl|}
		\int_{-\f t \epsilon}^{\f {2t}\epsilon} \big(1\wedge|u_3|^{-\beta(|\bl|- r )}\big)du_3 \int_{-\f t \epsilon}^{\f {2t}\epsilon} \big(1\wedge|u_2|^{-\beta r}\big)du_2  
		=\CO\left(\varepsilon^{2\beta|\bl|-4}\right).		
	\end{align*}	
	Thus, the claim follows since $\beta|\bl| > 1$ by assumption. 
\end{proof}

\begin{proposition}\label{prop:main_fdd}
  Let $0\leq t_1<\cdots<t_m\leq T$ and $\bl^1,\dots,\bl^k\in\N_0^n$. Then
  \begin{equation*}
	\begin{pmatrix}
	  \sqrt{\varepsilon}\int_0^{\frac{t_1}{\varepsilon}} H_{\bl^1}\left(z_u \right)\,du\\
	  \vdots\\
	  \sqrt{\varepsilon}\int_0^{\frac{t_m}{\varepsilon}} H_{\bl^1}\left(z_u\right)\,du
	\end{pmatrix}\oplus\cdots\oplus\begin{pmatrix}
	  \sqrt{\varepsilon}\int_0^{\frac{t_1}{\varepsilon}} H_{\bl^k}\left(z_u\right)\,du\\
	  \vdots\\
	  \sqrt{\varepsilon}\int_0^{\frac{t_m}{\varepsilon}} H_{\bl^k}\left(z_u\right)\,du
	\end{pmatrix}\Rightarrow\CN(0,\Lambda),
  \end{equation*} 
  where
  \begin{align*}
	\Lambda&\define\begin{pmatrix}
	  \Lambda^{1,1}&\cdots&\Lambda^{1,k}\\
	  \vdots & \ddots & \vdots\\
	  \Lambda^{k,1} & \cdots & \Lambda^{k,k}
	\end{pmatrix},\\
	\Lambda^{i,j}_{\alpha,\beta}&\define(t_\alpha\wedge t_\beta) \delta_{|\bl^i|,|\bl^j|}\int_0^\infty \Expec { H_{ \bl^i  }\left(z_u\right)  H_{ \bl^j  }\left(z_0\right)+ H_{ \bl^i  }\left(z_0\right) H_{ \bl^j  }\left(z_u\right)} \,du,\qquad \alpha,\beta=1,\dots,m.
  \end{align*}
\end{proposition}
\begin{proof}
Recall  $H_\bl (z_u) =I\big(\tau_u\hat{K}^{\otimes \bl} \big)$ from \cref{lem-decomposition}. Now \cref{lem:limit,lem:fourth_moment_applies} allow to conclude the claim by  an application of \cref{fourth-moment-theorem}.
\end{proof}
\begin{remark}
\label{polynomial-convergence}
This together with the lemmas below immediately leads to parts 1 and 2 of \cref{thm-A}  when $G_k$ are polynomial functions:
 $X^\epsilon\to \Upsilon W$ in finite-dimensional distributions.   The second claim follows for polynomials from the H\"older bounds proved below.
\end{remark}

The following $L^2$ estimate, which we will actually lift to an $L^p$ bound momentarily, plays a key r\^ole in proving the weak convergence in H\"older topology:
\begin{lemma}[$L^2$ H\"older bound]\label{lem:l2_bound}
   Let $0\leq s\leq t\leq T$. Let $G \in L^2(\R^n, N(0,\Sigma))$ be a real-valued function satisfying the fast chaos decay assumption with parameter $(2n-1)\hat{\Theta}$. If $\herm(G) > \beta^{-1}$, then 
  \begin{equation*}
	\left\|\sqrt{\varepsilon}\int_{\frac{s}{\varepsilon}}^{\frac{t}{\varepsilon}} G(y_r)\,dr\right\|_{L^2}\lesssim |t-s|^{\frac12}.
  \end{equation*}
\end{lemma}
\begin{proof}
  Owing to stationarity of $(y_t)_{t\geq 0}$, there is no loss of generality in assuming $s=0$. We expand the square as double integral and make use of the expansion \eqref{eq:normalized_hermite}:
  \begin{align}
	\Expec{\left(\int_0^{\frac{t}{\varepsilon}}G(y_u)\,du\right)^2}
	=\sum_{\substack{\bk,\bl\in\N_0^n \\ \vert \bk\vert , | \bl | \geq \herm(G)}} c_{\bk} c_{\bl} \int_0^{\frac{t}{\varepsilon}}\int_0^{\frac{t}{\varepsilon}} \Expec{ H_{\bk}(z_u) H_{\bl}(z_v) }\,du\,dv.\label{eq:lipschitz_estimate}
  \end{align}
  The series is absolutely summable,  thus one can exchange the order of summation and integration.
Recall that
  \begin{equation}\label{eq:normalized_decay}
	 \big|\expec{z_u\otimes z_v }\big| \leq \hat{\Theta}  \left(1 \wedge \vert u-v\vert^{-\beta}\right). 
  \end{equation} 
  Combining this with \cref{thm-diagramm-formulae}, we find 
  \begin{align*}
	\expec{ H_{\bk}(z_u) H_{\bl}(z_v)} \leq  \delta_{|\bk|,|\bl|}\sqrt{\bk !  \bl !} \hat{\Theta}^{ | \bk |} (2n-1)^{\vert \bk \vert}\left( 1 \wedge \vert u-v\vert^{- |\bk|\beta} \right).
  \end{align*}
  Inserting this estimate back into \eqref{eq:lipschitz_estimate}, we get
  \begin{align*}
	&\phantom{\leq}\Expec{\left(\int_0^{\frac{t}{\varepsilon}}G(y_u)\,du\right)^2} \\
	&\leq \sum_{\substack{\bk,\bl\in\N_0^n \\ \vert \bk\vert , | \bl | \geq \herm(G)}} | c_{\bl}| | c_{\bk}| \sqrt{ \bl !  \bk!} \hat{\Theta}^{ \f{| \bl| + | \bk|}{2}} ( 2n-1)^{ \f{| \bl| + | \bk|}{2}} \int_0^{\frac{t}{\varepsilon}}\int_0^{\frac{t}{\varepsilon}} \left( 1 \wedge \vert u-v\vert^{- \herm(G)\beta} \right)\,du\,dv\lesssim\frac{t}{\varepsilon},
  \end{align*}
  since $\herm(G) > \beta^{-1}$ and  $\sum_{\bl \in\N_0^n}  \vert c_{\bl} \vert  \hat{\Theta}^{ \f { | \bl |} {2}}  (2n-1)^{ \f {\vert \bl \vert}{ 2}} \sqrt{\bl !} < \infty$ by assumption. This concludes the proof.
\end{proof}
The following estimate will be used in \cref{proof-A}:
\begin{lemma}[$L^p$ H\"older bound]\label{lem:lp_holder}
 Let $0\leq s\leq t\leq T$. Let $G \in L^2(\R^n, N(0,\Sigma))$. If $\herm(G) > \beta^{-1}$ and $G$ satisfies the fast chaos decay assumption with parameter $(2n-1)(p-1)\hat{\Theta}+1$,  where $p>2$, then
  \begin{equation*}
	\left\|\sqrt{\varepsilon}\int_{\frac{s}{\varepsilon}}^{\frac{t}{\varepsilon}} G(y_r)\,dr\right\|_{L^p}\lesssim |t-s|^{\frac12}.
  \end{equation*}
\end{lemma}
\begin{proof}
  As in the proof of \cref{lem:l2_bound}, we may assume $s=0$ without any loss of generality. Recall that, by Gaussian hypercontractivity, $\|X\|_{L^p}\leq (p-1)^{\frac{|\bl|}{2}}\|X\|_{L^2}$ for any scalar random variable in the $|\bl|^\textup{th}$ Wiener chaos. It follows that
  \begin{align} 
  \left\|\int_{0}^{\frac{t}{\varepsilon}} G(y_r)\,dr\right\|_{L^p}&\leq\sum_{\bl\in\N_0^n}|c_{\bl}| (p-1)^{\frac{|\bl|}{2}} \left\|   \int_{0}^{\frac{t}{\varepsilon}} H_{\bl}(z_r) \,dr\right\|_{L^2}\nonumber\\
  &\lesssim \sqrt{\f{t}{\epsilon}} \sum_{\bl\in\N_0^n,}|c_{\bl}| (p-1)^{\frac{|\bl|}{2}} \sqrt{\bl!} \hat{\Theta} ^{\f{| \bl | }{2}}(2n-1)^{\f {| \bl | } {2}}.  \label{eq:lp_hypercontractivity}
  \end{align}  
The statement follows since the sum is finite by 
 the fast chaos decay assumption, c.f. \cref{def-Hermite-rank} \ref{it:fast_decay}.
\end{proof}

\subsection{Lifted Functional Central Limit Theorem}

The aim of this section is to prove the convergence of the iterated integrals. We begin with a definition inspired by \cite{Jacod-Shiryaev}. Let us write $\CF_t\define\sigma(W_s,s\leq t)$ for the filtration generated by the Wiener process driving $y$ through \eqref{eq:kernel}.

\begin{definition}\label{def:cond-dec-con}
We say that a function $G: \R^{n} \to \R$ with $ \E \left[ G(y_0) \right] =0$ satisfies the \textit{conditional decay condition} (with respect to $y$) if 
$$ \int_0^{\infty} \big\Vert \E \left[ G(y_s) | \mathcal{F}_0 \right] \big\Vert_{L^2(\Omega)}\, ds < \infty.$$
\end{definition}

As we shall see in the sequel, if $y\in\vol_n(\beta,\Theta)$, then any centered $G \in L^2(\R^n,N(0,\Sigma))$ with $\herm(G)  >2\beta^{-1}$ falls in the regime of \cref{def:cond-dec-con}.

Next, we recall a stability result for stochastic integrals, which is a weaker version of \cite[Theorem 2.2]{Kurtz-Protter}:
\begin{proposition}\label{prop:kurtz_protter}
	For each $k\in\N$, let $X^k$ and $M^k$ be stochastic processes adapted to a filtration $\CF^k$. If $(X^k,M^k)\Rightarrow(X,M)$ weakly in $\C\big([0,T],\R^{2N}\big)$ and, for each $k\in\N$, $M^k$ is an $\CF^k$ local martingale with $\sup_{k\in\N}\|M^k_T\|_{L^2(\Omega)}<\infty$, then
	\begin{equation*}
		\left(X^k,M^k,\int_0^\cdot X^k_s\,\otimes dM_s^k\right)\Rightarrow\left(X,M,\int_0^\cdot X_s\,\otimes dM_s\right)\qquad\text{in}\quad\C\Big([0,T],\R^{N}\oplus\R^N\oplus\R^N\otimes \R^N\Big).
	\end{equation*}
\end{proposition}
\begin{remark}
	The condition $\sup_{k\in\N}\|M^k_T\|_{L^2(\Omega)}<\infty$ is of course equivalent to Kurtz-Protter's famous \textit{uniformly controlled variation} (UCV) condition
	\begin{equation*}
		\sup_{k\in\N}\Expec{\Braket{M^k}_T}<\infty,
	\end{equation*}
	where $\braket{\cdot}$ denotes the quadratic variation.
\end{remark}

\begin{lemma}\label{lem:decomposition}
Let $y$ be any stationary stochastic process which is ergodic under the canonical time-shift. 
  For each $k \in \{1, \dots, N\}$, fix a function $G_k : \R^{n} \to \R$ satisfying the conditional decay condition with respect to $y$, the fast chaos decay assumption with parameter $(2n-1)\hat{\Theta}$, and having Hermite rank $\herm(G_k) > \beta^{-1}$. Set 
  \begin{equation*}
X^{\epsilon}_t = \left(\sqrt{\epsilon} \int_0^{\f t \epsilon} G_1(y_s) \,ds , \dots,  \sqrt{\epsilon} \int_0^{\f t \epsilon} G_N(y_s) \,ds\right),\qquad t\in[0,T].
  \end{equation*}
  Then we have the decomposition 
  $$ X^{\epsilon}_t = M^{\epsilon}_t + ( Z^{\epsilon}_0 - Z^{\epsilon}_{t}),$$
  where the respective $j^{\text{th}}$-components are defined by: 
  \begin{align*}
		M^{j,\epsilon}_t &\define\sqrt{\epsilon} \int_0^{\infty} \E\left[ G_j(y_s) | \mathcal{F}_{\f t \epsilon} \right] \,ds - \sqrt{\epsilon} \int_0^{\infty} \E\left[ G_j(y_s) | \mathcal{F}_0 \right]\,ds, \\
		Z^{j,\epsilon}_t &\define\sqrt{\epsilon} \int_{\f t \epsilon}^{\infty} \E\left[ G_j(y_s) | \mathcal{F}_{\f t \epsilon} \right]\, ds.
  \end{align*}
  In addition, the following hold:
  \begin{itemize}
  	\item For each $\varepsilon>0$, the process $M^\varepsilon$ is a martingale with respect to the rescaled filtration $\big(\mathcal{F}_{\f{ t}{\epsilon}}\big)_{t\geq 0}$.
  	\item We have $\sup_{\varepsilon\in(0,1]}\big\|M^\varepsilon_t\big\|_{L^2(\Omega)}<\infty$ for each $t\geq 0$.
  	\item For each $t\in [0,T]$, we have $\|Z^\varepsilon_t\|_{L^2(\Omega)}\to 0$.
  \end{itemize}
  
\end{lemma}
\begin{proof}
  Due to the conditional decay condition, the processes $M^{\epsilon}$ and $Z^{\epsilon}$ are well defined and, in particular, $(M^{\epsilon}_t)_{\varepsilon\in(0,1]}$ is uniformly bounded in $L^2$ for each $t\in[0,T]$. Indeed, splitting the first term and by the shift invariance of $y$ we have
  \begin{equation*}
  	\big\|M^{j,\varepsilon}_t\big\|_{L^2(\Omega)}\leq\sqrt{\varepsilon}\left\|\int_0^{\frac{t}{\varepsilon}}G_j(y_s)\,ds\right\|_{L^2(\Omega)}+2\sqrt{\varepsilon}\int_0^{\infty} \big\|\E\left[ G_j(y_s) | \mathcal{F}_0 \right]\big\|_{L^2(\Omega)}\,ds.
  \end{equation*}
  By \cref{lem:l2_bound} and the conditional decay condition, both terms on the right-hand side are uniformly bounded in $\varepsilon\in(0,1]$. The fact that $\|Z^\varepsilon_t\|_{L^2(\Omega)}\to 0$ for each $t\in[0,T]$ follows from the stationarity of $(y_t)_{t\geq0}$ in combination with the conditional decay condition on $G$ and Minkowski's integral inequality:
  \begin{equation*}
  	\|Z^{j,\varepsilon}_t\|_{L^2(\Omega)}=\sqrt{\varepsilon}\left\|\int_0^\infty\Expec{G_j(y_s)|\CF_0}\,ds\right\|_{L^2(\Omega)}\leq\sqrt{\varepsilon}\int_0^\infty\big\|\Expec{G_j(y_s)|\CF_0}\big\|_{L^2(\Omega)}\,ds\to 0.
  \end{equation*}
\end{proof}

We also need the following standard result, see e.g. \cite[Proposition 1.2]{Aldous1989}:
\begin{proposition}\label{prop:aldous}
	Let $(M^k)_{k\in\N}$ be a sequence of martingales. Suppose that:
	\begin{itemize}
		\item There is a continuous martingale $M$ such $M^k\fdd M$.
		\item For each $t\in [0,T]$, $\big(M_t^k\big)_{k\in\N}$ is uniformly integrable.
	\end{itemize}
	Then $M^k\Rightarrow M$ weakly in $\C\big([0,T],\R^N\big)$.
\end{proposition}

Combining \cref{prop:kurtz_protter,prop:aldous} with \cref{lem:decomposition}, we can prove the following central result:

\begin{proposition}\label{prop:weak_conv_cond_decay}
  Let $y$ be a stationary ergodic stochastic process and
  $X^{\epsilon}$ be as in \cref{lem:decomposition}, with the conditional decay condition with respect to $y$ in place.  Let   $W$ be a standard Wiener process and let
  $$ \Upsilon^2_{i,j} =  \int_0^{\infty} \Expec{ G_i(y_0) G_j(y_r) + G_i(y_r) G_j(y_0) } dr.$$
If $X^{\epsilon}\Rightarrow \Upsilon W$ in $\C\big([0,T],\R^N\big)$, as $\varepsilon\to 0$, then
  $$ \X^{\epsilon} = \big( X^{\epsilon}, \XX^\epsilon\big) \Rightarrow  \X$$
  in $\C\big([0,T],\R^N \oplus \R^N \otimes \R^N \big)$ where
  $$\X_{s,t}= \big( \Upsilon W_{s,t}, \quad  ( \Upsilon \otimes \Upsilon ) \WW_{s,t}+\Xi(t-s)\big)=\big(\Upsilon W_{s,t}, \quad (\Upsilon\otimes\Upsilon)\WW_{s,t}^{\text{It\^o}} + \Lambda(t-s)\big)$$ with  $\WW_{s,t}$ denoting the
  Stratonovich integral,  $\Upsilon$  the non-negative symmetric square root of $\Upsilon^2$,  and
  $$ \Xi_{i,j} = \f 1  2 \int_0^{\infty} \E\left[ G_i(y_0) G_j(y_r) -  G_i(y_r) G_j(y_0) \right] dr $$  
\end{proposition}
\begin{proof}
  Without loss of generality, we may assume $s=0$ by stationarity. By \cref{lem:decomposition} and an integration by parts, we have the decomposition
  \begin{align*}
	\int_0^t X^{\epsilon}_s \otimes dX^{\epsilon}_s &= \int_0^t M^{\epsilon}_s \otimes dX^{\epsilon}_s +   \int_0^t \left( Z^{\epsilon}_0 - Z^{\epsilon}_s \right) \otimes dX^{\epsilon}_s\\
	&= M^{\epsilon}_t \otimes X^{\epsilon}_t -\left(\int_0^t X^{\epsilon}_s \otimes dM^{\epsilon}_s\right)^T +   \int_0^t \left( Z^{\epsilon}_0 - Z^{\epsilon}_s \right) \otimes dX^{\epsilon}_s.
  \end{align*}
  Since $\|Z_t^\varepsilon\|_{L^2(\Omega)}\to 0$ for each $t\in [0,T]$ by \cref{lem:decomposition}, we see that $M^\varepsilon\to\Upsilon W$ in finite-dimensional distributions. Since $(M_t^\epsilon)$ is $L^2$ bounded on $[0,T]$ with uniform bound in $\epsilon$, Burkholder-Davis-Gundy inequality shows that $(X^\varepsilon,M^\varepsilon)_{\varepsilon\in(0,1]}$ is tight in $\C\big([0,T],\R^{2N}\big)$. Consequently, $Z_0^\varepsilon-Z^\varepsilon$ is also tight and converges to zero in probability. Since $X^\varepsilon=M^\varepsilon+Z_0^\varepsilon-Z^\varepsilon$, it follows that
  \begin{equation*}
  	\big(X^\varepsilon,M^\varepsilon,Z_0^\varepsilon-Z^\varepsilon\big)\Rightarrow\big(\Upsilon W, \Upsilon W, 0\big)\qquad\text{in}\quad\C\big([0,T],\R^{3n}\big).
  \end{equation*}
  By \cref{prop:kurtz_protter}, 
  \begin{align*}
  	M^{\epsilon}_t \otimes X^{\epsilon}_t -\left(\int_0^t X^{\epsilon}_s \otimes dM^{\epsilon}_s\right)^T&\Rightarrow(\Upsilon\otimes\Upsilon)(W_t \otimes W_t)-\left(\int_0^t (\Upsilon W)_s\otimes d(\Upsilon W)_s\right)^T \\
   &=\int_0^t (\Upsilon W)_s\otimes d(\Upsilon W)_s+t\Upsilon^2 = ( \Upsilon \otimes \Upsilon ) \WW_{0,t}+\frac{t}{2}\Upsilon^2
  \end{align*}
  weakly in $\C\big([0,T],\R^N\otimes\R^N\big)$. To conclude, we observe that
  \begin{align*}
   \int_0^t \left( Z^{\epsilon}_0 - Z^{\epsilon}_s \right) \otimes dX^{\epsilon}_s
	= \epsilon \int_0^{\f t \epsilon} \left(  \int_{0}^{\infty} \E\left[ G(y_r) | \mathcal{F}_{0} \right]\otimes  G(y_s)\,dr  -  \int_{s}^{\infty} \E\left[ G(y_r) | \mathcal{F}_{s} \right]\otimes G(y_s)\,dr \right)\,ds.
  \end{align*}
  By Birkhoff's ergodic theorem and stationarity, 
  for each fixed $t\in[0,1]$ these terms converge almost surely as $\varepsilon\to 0$ to 
  $$ \int_{0}^{\infty} \Expec{G(y_r) | \mathcal{F}_{0}}\,dr\otimes  t\Expec{G(y_0)}  -
   t \int_{0}^{\infty} \Expec{G(y_{r}) \otimes G(y_0)} \,dr .
  $$ 
  We used the shift invariance and ergocidity of the process. The first term is $0$ and the second term is 
  \begin{equation*}
	-t  \int_0^{\infty} \Expec{G(y_r)\otimes G(y_0)}\,dr.
  \end{equation*}
  To see that this convergence of $   \int_0^t \left( Z^{\epsilon}_0 - Z^{\epsilon}_s \right) \otimes dX^{\epsilon}_s$ is actually almost surely in $\C\big([0,1],\R^N\big)$, we apply \cref{lem:uniform_convergence} below with 
  \begin{align*}
  	f(t)&\define\int_0^t\left(\int_{0}^{\infty} \Expec{G(y_r) | \mathcal{F}_{0}}\otimes  G(y_s)\,dr  -  \int_{s}^{\infty} \Expec{G(y_r) | \mathcal{F}_{s}}\otimes G(y_s)\,dr \right)\,ds\\
  	&\phantom{\define}+t  \int_0^{\infty} \Expec{G(y_r)\otimes G(y_0)}\,dr.
  \end{align*}
Putting everything together, we have $$\int_0^t X^{\epsilon}_s \otimes dX^{\epsilon}_s\Rightarrow ( \Upsilon \otimes \Upsilon ) \WW_{0,t}+\f 12 {t}\Upsilon^2-t  \int_0^{\infty} \Expec{G(y_r)\otimes G(y_0)}\,dr,$$ 
the proposition follows from the identity $\Xi=\f 12 {t}\Upsilon^2-t  \int_0^{\infty} \Expec{G(y_r)\otimes G(y_0)}\,dr$.
\end{proof}

\begin{lemma}\label{lem:uniform_convergence}
  Let $f:\R_+\to\R$ be locally bounded. If $\lim_{\varepsilon\to 0}\varepsilon f\big(\frac{t}{\varepsilon}\big)= 0$ for any $t\in[0,1]$, then also
  \begin{equation*}
	\lim_{\varepsilon\to 0}\sup_{t\in[0,1]}\left|\varepsilon f\left(\frac{t}{\varepsilon}\right)\right|=0.
  \end{equation*}
\end{lemma}
We return to the representation $y_t=\int_\R K(t-u)\,dW_u$  where $K:\R\to\Lin[d]{n}$ and  set
$$  \bar{y}^{\tau}_t\define \int_{-\infty}^{\tau} K(t-u)\,dW_u, \qquad
\tilde{y}^{\tau}_t\define\int_{\tau}^{t} K(t-u)\,dW_u .
$$
Note that, for every $\tau>0$, $y_t = \bar{y}^{\tau}_t + \tilde{y}^{\tau}_t$. Moreover, $\bar{y}^{\tau}\in\mathcal{F}_{\tau}$, whereas $\tilde{y}^{\tau}$ is independent of $\mathcal{F}_{\tau}$. 
As the normalized process $z_t = D^{-\frac12}O^\top y_t$ is nothing but a linear transformation of $y$, we also have the decomposition $z_t =\bar{z}^{\tau}_t + \tilde{z}^{\tau}_t$, where
$$
\bar{z}^{\tau}_t = D^{-\frac12}O^\top \bar{y}^{\tau}_t , \quad \quad \quad \tilde{z}^{\tau}_t = D^{-\frac12}O^\top \tilde{y}^{\tau}_t.
$$
By independence, $\| \bar{z}^{\tau}_t \|^2_{L^2} + \|\tilde{z}^{\tau}_t \|^2_{L^2} =1$ necessarily holds.

\begin{lemma}\label{lem:z-bar-L2-decay}
For $t \geq \tau$, we have that
$$ \| \bar{z}^{\tau}_t \|_{L^2} \leq \hat{\Theta}  \left(1 \wedge |t-\tau|^{-\beta}\right).$$
\end{lemma}
\begin{proof}
For $j=1,\dots, n$ let $\bar{z}^{j,\tau}$ be the $j^\text{th}$ component of $\bar{z}^{\tau}$. We have that
\begin{equation*}
\Expec{ \big(  \bar{z}^{j,\tau}_t  \big)^2 }  = \int_{-\infty}^{-(t-\tau)} \braket{\hat{K}_j(-u),\hat{K}_j(-u)}\, du \leq \hat{\Theta}  \left( 1\wedge \vert t- \tau \vert^{- 2 \beta}\right),
\end{equation*}
as required.
\end{proof}

\begin{lemma}\label{lem-condional-expec}
  Let $\bl \in \N_0^n$ be a multi-index. Then, for each $ t \geq \tau$,
  $$
  \big \Vert \E \left[H_{\bl}(z_t)\,|\,\mathcal{F}_{\tau} \right] \big \Vert_{L^2} \leq \sqrt{ \bl !}  \hat{\Theta} ^{\f {|\bl|} {2}}(2n -1)^{ \f{\vert \bl \vert} {2}}  \left( 1 \wedge  |t-\tau|^{- \f{\vert \bl \vert}{2} \beta } \right).
  $$
\end{lemma}
\begin{proof}
We use the following well-known expansion formula for Hermite polynomials. For any $a,b \in \R$ with $a^2+b^2=1$:
  \begin{equation*}
	H_m(ax+by) = \sum_{j=0}^m  \binom{m}{j} a^{j} b^{m-j} H_{   j}(x) H_{m-   j}(y)\qquad\forall\,x,y\in\R.
  \end{equation*}
  Taking $a=\| \bar{z}^{\tau}_t \|_{L^2} $ and $b= \|\tilde{z}^{\tau}_t \|_{L^2}$, we obtain
  \begin{align*}
	H_{\bl}(z_t)&=\prod_{k=1}^{n} H_{\bl_k}(z^k_t) = \prod_{k=1}^{n} \sum_{j_k=0}^{\bl_k}  \binom{\bl_k}{j_k} \big\Vert \bar{z}_t^{k,\tau}\big\Vert_{L^2}^{j_k} \big\Vert \tilde{z}_t^{k,\tau} \big\Vert_{L^2}^{\bl_k -j_k} H_{j_k}\left(\f{\bar{z}^{k,\tau}_t} {\big\Vert \bar{z}^{k,\tau}_t \big\Vert_{L^2}} \right) H_{\bl_k-   j_k}\left( \f {\tilde{z}^{k,\tau}_t} { \big\Vert \tilde{z}_t^{k,\tau} \big\Vert_{L^2}}\right)\\
	&= \sum_{j_1 =0}^{\bl_1} \dots \sum_{j_n =0}^{\bl_n}  \prod_{k=1}^n \binom{\bl_k}{j_k} \big\Vert \bar{z}_t^{k,\tau}\big\Vert_{L^2}^{j_k} \big\Vert \tilde{z}_t^{k,\tau} \big\Vert_{L^2}^{\bl_k -j_k} H_{j_k}\left(\f{\bar{z}^{k,\tau}_t} {\big\Vert \bar{z}^{k,\tau}_t \big\Vert_{L^2}} \right) H_{\bl_k-j_k}\left( \f {\tilde{z}^{k,\tau}_t} { \big\Vert \tilde{z}_t^{k,\tau} \big\Vert_{L^2}}\right).
  \end{align*}
  Since $\tilde{z}^{k,\tau}_t$ is independent of $\mathcal{F}_{\tau}$,  after taking the conditional expectation all terms vanish except those with $j_k=\bl_k$, thus, the sum reduces to one term:
  \begin{align*}
	&\phantom{=}\E \left[ H_{\bl}(z_t)\,|\,\mathcal{F}_{\tau} \right] \\
	&=  \sum_{j_1 =0}^{\bl_1} \dots \sum_{j_n =0}^{\bl_n}  \prod_{k=1}^n \binom{\bl_k}{j_k} \big\Vert \bar{z}_t^{k,\tau}\big\Vert_{L^2}^{j_k} \big\Vert \tilde{z}_t^{k,\tau} \big\Vert_{L^2}^{\bl_k -j_k} H_{j_k}\left(\f{\bar{z}^{k,\tau}_t} {\big\Vert \bar{z}^{k,\tau}_t\big\Vert_{L^2}} \right) \E \left[  H_{\bl_k -   j_k}\left( \f {\tilde{z}^{k,\tau}_t} {\big\Vert \tilde{z}^{k,\tau}_t \big\Vert_{L^2}}\right)\,\middle|\,\mathcal{F}_{\tau} \right]\\
	&=     \prod_{k=1}^n \big\Vert \bar{z}^{k,\tau}_t \big\Vert_{L^2}^{\bl_k}  H_{\bl_k}\left(\f{\bar{z}^{k,\tau}_t} {\big\Vert \bar{z}^{k,\tau}_t \big\Vert_{L^2}} \right).
  \end{align*}
  The asserted estimate follows by an application of \cref{thm-diagramm-formulae} and the diagram formula to the term $\E\big[H_{\bl}\bigl(\bar{z}^{k,\tau}_t / \big\Vert \bar{z}^{k,\tau}_t \big\Vert_{L^2} \big)\big]$. The multiplicative factors just cancel. The required estimate follows from \cref{lem:z-bar-L2-decay} and Cauchy-Schwarz. 
  \end{proof}

\begin{proposition}\label{prop:volterra_cond_decay}
  Let $y\in\vol_n(\beta,\Theta)$. Then any function $G \in L^2(\R^n,N(0,\Sigma))$, which satisfies $\herm(G) >2\beta^{-1}$ and the fast chaos decay condition with parameter $\hat{\Theta} (2n-1) +1$, also satisfies the conditional decay condition of \cref{def:cond-dec-con}.
\end{proposition}
\begin{proof}
Using the Hermite expansion of $G$ and \cref{lem-condional-expec}, we compute
\begin{equation*}
\big\Vert \E \left[ G(y_s) | \mathcal{F}_0 \right] \big\Vert_{L^2} 
\leq  \sum_{\bl \in \N_0^n, \vert \bl \vert \geq\herm(G)} \vert c_{\bl} \vert \big \Vert \E \left[  H_{\bl}(z_s) | \mathcal{F}_0 \right] \big \Vert_{L^2}
\leq  \sum_{\bl \in \N_0^n, \vert \bl \vert \geq\herm(G)} \vert c_{\bl} \vert \sqrt{ \bl!} \hat{\Theta} ^{\f {|\bl|}{2}} (2n-1)^{\f {\vert \bl \vert} {2}} \left(1 \wedge s^{- \frac{|\bl|}{2} \beta} \right).
\end{equation*} 
Since $ \sum_{\bl \in \N_0^n} \vert c_{\bl} \vert \sqrt{ \bl!} \hat{\Theta}^{ \f {|\bl|}{2}}(2n-1)^{\f {\vert \bl \vert} {2}} < \infty$,
\begin{align*}
\int_0^{\infty} \big\Vert \E \left[ G(y_s) | \mathcal{F}_0 \right] \big\Vert_{L^2}\,ds
 \leq \sum_{\bl \in \N_0^n, \vert \bl \vert \geq\herm(G)} \vert c_{\bl} \vert \sqrt{ \bl!} \hat{\Theta} ^{\f {|\bl|}{2}}(2n-1)^{\f {\vert \bl \vert} {2}}
 \sup_{\vert \bl \vert \geq\herm(G)} \int_0^{\infty} \left(1 \wedge s^{-\f {\vert \bl \vert}{2} \beta} \right)\,ds< \infty.
\end{align*}
The finiteness of the integral follows from the assumption that $\herm(G)>2\beta^{-1} $.
\end{proof}

\begin{lemma}\label{lem-tightness-iterated-hermite-poly}
  Let $\bl, \bk \in \N_0^n$ be multi-indices such that $ \vert \bl \vert \wedge \vert \bk \vert > \beta^{-1}$. Then
  \begin{align*}
	\Expec{\left( \epsilon \int_{\f s \epsilon}^{\frac{t}{\varepsilon}} \int_{\f s \epsilon}^r H_{\bl}(z_r) H_{\bk}(z_u)\,du\,dr \right)^2} \lesssim \bl ! \bk! \hat{\Theta}^{|\bl | + | \bk|}(4n -1)^{\vert \bl \vert + \vert \bk \vert} \vert t-s\vert^2.
  \end{align*}
\end{lemma}
\begin{proof}
  By stationarity we may assume $s=0$. We certainly have
  \begin{equation}\label{eq:double_integral_bound}
	\Expec{\left( \int_0^{\frac{t}{\varepsilon}} \int_0^r H_{\bl}(z_r) H_{\bk}(z_u)\,du\,dr\right)^2} \leq \int_{[0, \f t \epsilon]^4}  \Big \vert \Expec{ H_{\bl}(z_{s_1}) H_{\bl}(z_{s_2}) H_{\bk}(z_{s_3}) H_{\bk}(z_{s_4}) } \Big \vert\,d^4s.
  \end{equation}
  Recall that $z_s\sim N(0,\id)$ for any $s\geq 0$. Consequently, \cref{thm-diagramm-formulae}  furnishes  the estimate 
  \begin{align*}
	\Big|\Expec{ H_{\bl}(z_{s_1}) H_{\bl}(z_{s_2}) H_{\bk}(z_{s_3}) H_{\bk}(z_{s_4}) }\Big| \leq \hat{\Theta}^{|\bl|+|\bk|}\sum_{\G\in\Gamma_{\bm}} \prod_{\substack{i,j=1\\i< j}}^{4}\Big(1 \wedge \vert s_i - s_j \vert^{ - \beta \gamma_{i,j}(\G)}\Big),
  \end{align*}
  where $\bm\define(\bl,\bl,\bk,\bk)\in\N_0^{4n}$ is obtained by simple concatenation of $\bl$ and $\bk$. Recall, $\Gamma_{\bl}$ denotes the set of graphs of complete pairings with the $k^{\text{th}}$ node having exactly $\bl_k$ edges.
For a graph $\G\in\Gamma_{\bl}$, we write $\gamma_{i,j}(\G)$ for the number of edges between the nodes $i$ and $j$. For any given $\G\in\Gamma_{\bm}$, $\gamma_{i,j}(\G)\in\N_0$ satisfy 
  \begin{alignat*}{3} 
	\gamma_{1,2}(\G)+\gamma_{1,3}(\G)+\gamma_{1,4}(\G)&=|\bl|,
	&\qquad \gamma_{1,2}(\G)+\gamma_{2,3}(\G)+\gamma_{2,4}(\G)&=|\bl|,\\
	\gamma_{1,3}(\G)+\gamma_{2,3}(\G)+\gamma_{3,4}(\G)&=|\bk|,
	&\qquad \gamma_{1,4}(\G)+\gamma_{2,4}(\G)+\gamma_{3,4}(\G)&=|\bk|.
  \end{alignat*}
  Owing to \eqref{eq:number_graphs}, the proof is concluded upon verifying that 
  \begin{equation}\label{eq:int_finite}
	\int_{[0,\frac{t}{\varepsilon}]^4}\prod_{\substack{i,j=1\\i< j}}^{4}\Big(1 \wedge \vert s_i - s_j \vert^{ - \beta \gamma_{i,j}(\G)}\Big)\,d^4s\lesssim\left(\frac{t}{\varepsilon}\right)^2,
  \end{equation}
  uniformly in $\G\in\Gamma_{\bm}$. To this end, let us fix such a graph and henceforth suppress it in our notation. We abbreviate $f(s)\define 1\wedge |s|^{-\beta}$, $\nu_k\define\sum_{i=1}^{k-1}\gamma_{i,k}$, and $\tau_k\define\sum_{i=k+1}^4\gamma_{k,i}$. Then
  \begin{equation*}
	\int_{[0,\frac{t}{\varepsilon}]^4}\prod_{i<j}f(s_i-s_j)^{\gamma_{i,j}}\,d^4s=\int_{[0,\frac{t}{\varepsilon}]^3}\left(\int_0^{\frac{t}{\varepsilon}}\prod_{j=2}^4 f(s_1-s_j)^{\gamma_{1,j}}\,ds_1\right)\prod_{\substack{i=2\\i<j}}^4f(s_i-s_j)^{\gamma_{i,j}}\,d^3s.
  \end{equation*}
  An application of H\"older's inequality shows that
  \begin{equation*}
	\int_0^{\frac{t}{\varepsilon}}\prod_{j=2}^4 f(s_1-s_j)^{\gamma_{1,j}}\,ds_1\leq\sup_{u\in[0,\frac{t}{\varepsilon}]}\int_0^{\frac{t}{\varepsilon}}f(s_1-u)^{\tau_1}\,ds_1\leq 2\int_0^{\frac{t}{\varepsilon}}f(s_1)^{\tau_1}\,ds_1.
  \end{equation*}
  Iterating this argument, we find
  \begin{equation*}
	\int_{[0,\frac{t}{\varepsilon}]^4}\prod_{i<j}f(s_i-s_j)^{\gamma_{i,j}}\,d^4s\leq 8\prod_{k=1}^4\int_0^{\frac{t}{\varepsilon}}f(s)^{\tau_k}\,ds.
  \end{equation*}
  Similarly, it holds that
  \begin{equation*}
	\int_{[0,\frac{t}{\varepsilon}]^4}\prod_{i<j}f(s_i-s_j)^{\gamma_{i,j}}\,d^4s\leq 8\prod_{k=1}^4\int_{0}^{\frac{t}{\varepsilon}}f(s)^{\nu_k}\,ds,
  \end{equation*}
  whence
  \begin{equation*}
	\left(\int_{[0,\frac{t}{\varepsilon}]^4}\prod_{i<j}f(s_i-s_j)^{\gamma_{i,j}}\,d^4s\right)^2\lesssim\prod_{k=1}^4\int_0^{\frac{t}{\varepsilon}}f(s)^{\tau_k}\,ds\int_0^{\frac{t}{\varepsilon}}f(s)^{\nu_k}\,ds.
  \end{equation*}
  Since $\tau_k+\nu_k\geq|\bl|\wedge|\bk|>\beta^{-1}$ for each $k=1,\dots,4$, it is easy to see that the right-hand side is $\lesssim (t/\varepsilon)^4$. This verifies \eqref{eq:int_finite} and the proof is complete.
\end{proof} 

\begin{proposition}\label{prop:iterated-tightness}
  Let $p\geq 2$. Let $G_1,G_2 \in L^2(\R^n,N(0,\Sigma))$ satisfy the fast chaos decay condition with parameter $\hat{\Theta}(4n-1)(p-1)+1$. Then
\begin{equation*}
  \left \Vert \epsilon \int_{\f s \epsilon}^{\frac{t}{\varepsilon}} \int_{\f s \epsilon}^r G_1(y_r) G_2(y_u)\,du\,dr \right \Vert_{L^p} \lesssim  \vert t-s\vert
\end{equation*}
\end{proposition}
\begin{proof}
We may again assume $s=0$ by stationarity.  We make use of \eqref{eq:hermite_expansion} and the hypercontractivity estimate already employed in \cref{lem:lp_holder}:
\begin{align*}
\left \Vert \int_{\f s \epsilon}^{\frac{t}{\varepsilon}} \int_{\f s \epsilon}^r G_1(y_r) G_2(y_u)\,du\,dr \right \Vert_{L^p} &\leq \sum_{\bl \in \N_0^n}  \sum_{\bk \in \N_0^n} \big\vert c^1_{\bl} c^2_{\bk} \big\vert \left \Vert \int_{\f s \epsilon}^{\frac{t}{\varepsilon}} \int_{\f s \epsilon}^r H_{\bl}(z_r)  H_{\bk}(z_u) \, du\, dr \right \Vert_{L^p}\\
&\leq \sum_{\bl \in \N_0^n}  \sum_{\bk \in \N_0^n} \big\vert c^1_{\bl} c^2_{\bk} \big\vert (p-1)^{\f {\vert \bl \vert + \vert \bk \vert}{2}}  \left \Vert \int_{\f s \epsilon}^{\frac{t}{\varepsilon}} \int_{\f s \epsilon}^r H_{\bl}(z_r)  H_{\bk}(z_u) \, du\, dr \right \Vert_{L^2}\\
&\lesssim \f t \epsilon\sum_{\bl \in \N_0^n}  \sum_{\bk \in \N_0^n} \big\vert c^1_{\bl} c^2_{\bk} \big\vert (p-1)^{\f {\vert \bl \vert + \vert \bk \vert} {2}}   \sqrt{ \bl! \bk!} \hat{\Theta}^{\f{ |\bk| + |\bl|}{2}}(4n-1)^{\f{\vert \bk \vert + \vert \bl \vert}{2}}.
\end{align*}
Here, the last estimate relies on \cref{lem-tightness-iterated-hermite-poly}. The fast chaos decay assumption precisely states that the double sum on the right-hand side is finite. This proves the claim.
\end{proof}

\subsection{Proof of \cref{thm-A}}\label{proof-A}
With the preparation work earlier,  together with  \cref{prop:weak_convergence} below concerning convergence in the rough path topology,
we can now prove \cref{thm-A}.

\begin{proof}[of \cref{thm-A}] 
Let  $X^\epsilon$ be as in  \eqref{eq:finite_dim}. In \cref{polynomial-convergence} we have seen the convergence of $X^\epsilon\to \Upsilon W$ in finite-dimensional distributions if $G_k$ are polynomials.
Let $M\in\N$ and $G \in L^2(\R^n, N(0,\Sigma)) $. We denote the function obtained by truncating the Hermite expansion of $G$ at level $M$ by
\begin{equation*}
  G^M(y_s) \define\sum_{\substack{\bl\in\N_0^n\\|\bl|\leq M}}c_{\bl}H_{\bl}(z_s).
\end{equation*}
  Combining \cref{lem-decomposition,prop:main_fdd}, it follows that, for each $M\in\N$,
  \begin{equation*}
	X^{\epsilon, M}\define \left( \sqrt{\epsilon} \int_0^{\frac{t}{\epsilon}} G_1^M(y_s) ds , \dots, \sqrt{\epsilon} \int_0^{\frac{t}{\epsilon}} G_N^M(y_s) ds \right)_{t\in[0,1]}\fdd \Upsilon_M\, W,
  \end{equation*}
  where now
  \begin{equation*}
	 (\Upsilon_M^2)_{i,j}\define \int_0^{\infty}\Expec{G^M_{i}(y_s)  G^M_{j}(y_0) +  G^M_{i}(y_0)  G^M_{j}(y_s)} \,ds.
  \end{equation*}
 The convergence  in the finite-dimensional distributions follows from the Portemanteau theorem,
  see e.g. \cite[Theorem 3.2]{Billingsley}:  we only need to verify the convergence of each component in   probability (since we also have  $\Upsilon_M \to \Upsilon$ and therefore $\Upsilon_M W\Rightarrow\Upsilon W$).

  It remains to send $M\to\infty$. For this we observe that, for each $t\in[0,1]$ and each $i=1,\dots,N$,
  \begin{align*}
	\Expec{\left(\sqrt{\varepsilon}\int_0^{\frac{t}{\varepsilon}}G_i(y_s)\,ds-\sqrt{\varepsilon}\int_0^{\frac{t}{\varepsilon}}G_i^M(y_s)\,ds\right)^2}&=\varepsilon\sum_{\substack{\bk,\bl \in\N_0^{n}\\|\bk|,|\bl|>M}}c_{\bl}c_{\bk}\int_0^{\frac{t}{\varepsilon}}\int_0^{\frac{t}{\varepsilon}}\Expec{H_{\bl}(y_u)H_{\bk}(y_v)}\,du\,dv\\
	&\lesssim\sum_{\substack{\bk,\bl\in\N_0^{n}\\|\bk|,|\bl|>M}} \delta_{ | \bl |, | \bk |} c_{\bl}c_{\bk} \sqrt{\bl !  \bk !}  \hat{\Theta}^{| \bl |}( 2 n -1)^{ \vert \bl \vert}\to 0
  \end{align*}
  uniformly in $\varepsilon>0$ as $M\to\infty$. 

  Fix $\gamma<\frac12-\big(\min_k p_k\big)^{-1}$. 
The weak convergence of $X^\epsilon$ to $\Upsilon W$  in $\C^{\gamma}([0,T],\R^N)$ follows from  the tightness of the left-hand side of \eqref{eq:main_fdd_convergence} in $\C^{\gamma}([0,T],\R^N)$ (upon slightly decreasing $\gamma$). Together with the fact that finite-dimensional distributions uniquely determine Radon  measures on $\C^{\gamma}([0,T],\R^N)$, c.f. \cite[Ex. 7.14.79]{Bogachev2007}, this follows similarly as in \cref{prop:weak_convergence} below. We conclude the asserted weak convergence. 
  
  It is only left to lift this limit theorem to the iterated integrals. 
 By \cref{prop:volterra_cond_decay} and the convergence of $X^\epsilon$ to $\Upsilon W$, the process $X^\epsilon$, with $y\in\vol_n(\beta,\Theta)$ and with 
 $G _k\in L^2(\R^n,N(0,\Sigma))$ satisfying $\herm(G) >2\beta^{-1}$ and the fast chaos decay condition with parameter $\hat{\Theta} (2n-1) +1$, also satisfies the conditions
of \cref{prop:weak_conv_cond_decay}. Applying the latter we conclude the convergence 
  \begin{equation*}
	\X^{\epsilon}=\left( X^{\epsilon},\XX^{\epsilon} \right)\fdd\big(\Upsilon W, ( \Upsilon \otimes \Upsilon ) \WW_{s,t}+\Xi(t-s)\ \big).
  \end{equation*}
  Furthermore,   $\|X^\epsilon_{s,t}\|_{L^p}\lesssim |t-s|^{\f 12}$ by \cref{lem:lp_holder}
  and  $\|\X^\epsilon_{s,t}\|_{L^p}\lesssim |t-s|^{\f 12}$ by
   \cref{prop:iterated-tightness}   for every $\gamma < \f 1 2 - ( \min_k p_k )^{-1}$, with bounds independent of $\epsilon$. Thus, $
  \sup_{\epsilon \in (0, 1] }\Expec{|\X^{\epsilon}|_{\C^\gamma\oplus\C^{2\gamma}}}<\infty
  $, and an application of  \cref{prop:weak_convergence} concludes the proof.
\end{proof}

\section{Rough Homogenization}
Consider the family of ordinary differential equations (ODEs)
\begin{equation}\label{eq:homogenization}
dx_t^\varepsilon=\frac{1}{\sqrt{\varepsilon}}f\big(x_t^\varepsilon,y_{\frac{t}{\varepsilon}})\,dt,\qquad x_0^\varepsilon=x_0,\qquad t\in[0,T],
\end{equation}
depending on a scale parameter $\varepsilon>0$. Let $D_{\bl} f$ denote the derivative of $f$ in the $x$-variable, of order $\bl=(\bl_1, \dots, \bl_d)$. We also write $D_x f$ for the full Jacobian in the $x$-variable. We shall interpret \eqref{eq:homogenization} as a Banach space-valued \textit{rough differential equation} (RDE), see \cref{sec:rough_path} for background.

\subsection{Casting the Homogenization Problem as RDE}

Let us explain how to cast the ODE \eqref{eq:homogenization} as an infinite-dimensional RDE. 
To this end, we set
\begin{equation}\label{eq:x_eps}
  X^{\epsilon} : [0,T] \to \C_b^\alpha(\R^d, \R^d),\qquad t \mapsto \frac{1}{\sqrt{\varepsilon}}\int_0^t f( \cdot, y_{\frac{s}{\varepsilon}})\,ds.
\end{equation}
Thus, $X^{\epsilon}$ is a curve in $\C_b^\alpha(\R^d, \R^d)$. We denote its canonical rough path lift by $\X^{\epsilon}=(X^\varepsilon,\XX^\varepsilon)$, where
\begin{equation}\label{eq:xx_eps}
  \XX^{\varepsilon}_{s,t}=\int_s^t \big(X^{\varepsilon}_r-X^{\varepsilon}_s\big)\otimes dX^{\varepsilon}_r= \f{1}{\epsilon}\int_s^t\left(\int_s^r f(\cdot,
   y_ {\f u \epsilon }). \,du\right)\otimes f(\cdot\cdot,  y_ {\f r \epsilon }))\,dr.
\end{equation}
This notation has to be understood in terms of the embedding $\C_b^\alpha(\R^d,\R^d)\otimes\C_b^\alpha(\R^d,\R^d)\hookrightarrow\C_b^\alpha(\R^d\times\R^d,\R^d\otimes\R^d)$, see \cref{lem:tensor_product}. The function  $ \XX^{\varepsilon}_{s,t}:\R^d\times\R^d\to\R^d\otimes\R^d$ is given by
\begin{equation*}
\XX^{\varepsilon}_{s,t}(x_1,x_2)=\f{1}{\epsilon}\int_s^t\left(\int_s^r f(x_1,   y_ {\f u \epsilon }))\,du\right)\otimes f(x_2,  y_ {\f r \epsilon }))\,dr.
\end{equation*}
Furthermore, we define 
\begin{equation}\label{eq:evaluation_functional}
\delta : \R^d \to \L\big( \C_b^\alpha(\R^d,\R^d),\R^d\big),\qquad  \hbox{with} \quad \delta_x(f) = f(x).
\end{equation}
Occasionally, we shall also use the notation $\delta(x)\define\delta_x$. The operator $\delta$ is sufficiently regular as the next lemma shows: 
\begin{lemma}\label{lem:delta}
For $\alpha\in(0,\infty)\setminus\N$, the operator \eqref{eq:evaluation_functional} inherits the regularity of the image space:
\begin{equation*}
  \delta  \in \C_b^\alpha\Big(\R^d,\L\big( \C_b^\alpha(\R^d,\R^d),\R^d\big)\Big).
\end{equation*}
\end{lemma}
\begin{proof}
	For simplicity let us only consider $\alpha\in(1, 2)$. All other cases can be treated similarly. We claim that $D\delta_x\in\L\big(\R^d,\L\big(\C_b^\alpha(\R^d,\R^d),\R^d\big)\big)$ acts as
	\begin{equation}\label{eq:derivative_delta}
		D\delta_x(h)(f)=Df(x) h,\qquad f\in\C_b^\alpha\big(\R^d,\R^d\big),\;h\in\R^d.
	\end{equation}
	In fact,
  \begin{align*}
	\big\|\delta_{x+h}-\delta_x-D\delta_x(h) \big\|_{\L(\C_b^\alpha(\R^d,\R^d),\R^d)}&=\sup_{\|f\|_{\C_b^{\alpha}(\R^d,\R^d)}\leq 1}\big|f(x+h)-f(x)-Df(x)h\big|\\
	&\leq |h|\sup_{\|f\|_{\C_b^{\alpha}(\R^d,\R^d)}\leq 1}\int_0^1 \big|Df(x+\varsigma h)-Df(x)\big|\,d\varsigma\leq|h|^{\alpha}.
  \end{align*}
\end{proof}

Next, we write \eqref{eq:homogenization} as a rough differential equation:
\begin{lemma}\label{lem:ode} Fix a $f\in\C_b^{3,0+}(\R^d\times\R^n,\R^d)$ and a deterministic path $y\in\C^{\bar{\gamma}}\big([0,\frac{T}{\varepsilon}],\R^n\big)$ for some $\bar\gamma>0$.
 Then, for any $\varepsilon>0$, the RDE 
\begin{equation}\label{eq:rde_formula}
  dx^{\epsilon}_t =\delta(x^{\epsilon}_t)\,d\X^{\epsilon}_t
\end{equation}
driven by the rough path defined in \eqref{eq:x_eps} and \eqref{eq:xx_eps} is well posed. Furthermore, the unique global solution to \eqref{eq:homogenization} solves the RDE \eqref{eq:rde_formula}.
\end{lemma}
\begin{proof}
Well-posedness of \eqref{eq:rde_formula} follows immediately from \cref{prop:rde_existence,lem:delta}. For the last part set
  \begin{align*}
	\Xi_{s,t}^\varepsilon &\define \delta(x_s^\varepsilon)(X_t^\varepsilon-X_s^\varepsilon)+\big(\delta(x_s^\varepsilon)\odot D\delta(x_s^\varepsilon)\big) \XX^\varepsilon_{s,t},\\
	 \bar \Xi_{s,t}^\varepsilon &\define f(x_s^\varepsilon,y_{\frac{s}{\varepsilon}})(t-s),
  \end{align*}
  where $D\delta$ acts by \eqref{eq:derivative_delta}. By the sewing lemma, it suffices to show that,  for some $\eta>1$,
  \begin{equation*}
	\big|\Xi_{s,t}^\varepsilon-\bar \Xi_{s,t}^\varepsilon\big|\lesssim |t-s|^\eta,
  \end{equation*}
   as it implies the integrals obtained by $ \Xi_{s,t}^\varepsilon$  and $\bar \Xi_{s,t}^\varepsilon$ are the same, see  \cite{Gubinelli-lemma,Feyel-delaPradelle,Friz-Hairer}. Note that, since $\varepsilon>0$ is fixed, the prefactor in this bound is allowed to diverge as $\varepsilon\to 0$. Remember that, by assumption, we have $y\in\C^{\bar{\gamma}}\big([0,\frac{T}{\varepsilon}],\R^m\big)$ with probability $1$ and $f\in\C_b^{3,0+}$. It follows that, for some $\eta>0$,
  \begin{align*}
	\big|\Xi_{s,t}^\varepsilon-\bar \Xi_{s,t}^\varepsilon\big|&\leq\frac{1}{\sqrt{\varepsilon}}\int_s^t\big|f\big(x_s^\varepsilon,y_{\frac{r}{\varepsilon}}\big)-f\big(x_s^\varepsilon,y_{\frac{s}{\varepsilon}}\big)\big|\,dr+\frac{1}{\varepsilon}\int_s^t\left(\int_s^r \big|D_x f\big(x_s^\varepsilon,y_{\frac{r}{\varepsilon}}\big)\big|\big|f\big(x_s^\varepsilon,y_{\frac{u}{\varepsilon}}\big)\big|\,du\right)\,dr\\
	&\lesssim  |t-s|^{1+\eta},
  \end{align*}
  as required.
\end{proof}

\subsection{Statement of the Result}

We begin with a simple lemma:
\begin{lemma}\label{lem:hermite_rank_derivative}
  Let $f:\R^d\times\R^n\to\R^d$ be differentiable such that $f(x,\cdot)\in L^2\big(\R^n,N(0,\Sigma);\R^n\big)$ for every $x\in\R^d$. Then, if $\sup_{|x|\leq R}\|D_xf(x,\cdot)\|_{L^2(\R^n,N(0,\Sigma))}<\infty$   for each $R>0$, 
  \begin{equation*}
	 \inf_x\herm\big(D_xf(x,\cdot)\big)\geq\inf_x\herm\big(f(x,\cdot)\big).
   \end{equation*} 
\end{lemma}
\begin{proof} 
  If $\braket{f(x,\cdot),H_{\bl}}=0$ for some multi-index $\bl\in\N_0^n$, then also
$	\Braket{D_x f(x,\cdot), H_{\bl}}=D_x\Braket{f(x,\cdot),H_{\bl}}=0$.
\end{proof}

We define the Hermite rank of $f(x,\cdot)$ to be
\begin{equation*}
	\herm(f(x,\cdot))=\inf_{i\in\{1, \dots, n\}} \herm (f_i(x, \cdot)).
\end{equation*}
Let $c_{\bl}=(c_{\bl}^1, \dots, c_{\bl}^n)$ and $c_{\bl}^{\bk}=(c_{\bl}^{1, \bk}, \dots, c_{\bl}^{n, \bk})$ where for $i=1, \dots, n$,
\begin{equation*}
	f_i(x, \cdot) =\sum_{\bl\in \N^n_0,  |\bl|\ge \herm(f_i(x, \cdot)) }   c_{\bl}^i(x)H_{\bl}, \qquad 
D_{\bk}f_i(x, \cdot) =\sum_{\bl\in \N^n_0,  |\bl|\ge \herm(D_{\bk} f_i(x, \cdot)) }   c_{\bl}^{i,\bk}(x)H_{\bl}.
\end{equation*}

\begin{condition}\label{cond:homogenization}
	Let $\gamma \in \left( \f 1 3, \f 1 2\right)$ and $p>\f{2}{1-2 \gamma}$. We impose the following conditions on the data in \eqref{eq:homogenization}:
	\begin{enumerate}
		\item The fast process $y$ has almost sure sample paths in $\C^{0+}\big([0,T],\R^n\big)$. Moreover, there are $\beta,\Theta>0$ such that $y\in\vol_n(\beta,\Theta)$.
		\item $f\in\C_b^{3,0+}\big(\R^d\times\R^n,\R^d\big)$.
		 For each $j=1,\dots,d$, $0\le |\bk|\le3$,  we have  $\inf_x\herm\big(f_j(x,\cdot)\big)>\f 2 \beta$,
		 $D_{\bk} f_j(x,\cdot)\in L^2\big(\R^n,N(0,\Sigma)\big)$, and 
		\begin{equation}\label{supremum-norm-1}
		\sum_{\bl\in\N_0^n}\sup_{x\in B_R}|c^{j,\bk}_{\bl}(x)|(4n-1)^{\frac{|\bl|}{2}} (p-1)^{\f{ |\bl|} {2}} \sqrt{\bl!} \hat{\Theta}^{\f{ \vert \bl \vert} {2}}  <\infty\qquad\forall\,R>0
		\end{equation}
		where $c^{j,\bk}$ are the coefficients in the Hermite expansion of $D_{\bk} f_j$. 
	\end{enumerate}
\end{condition}

 \begin{lemma}
	The condition \eqref{supremum-norm-1} holds if for some $\theta>\hat{\Theta}(4n-1)(p-1)$ and every $j=1,\dots,d$, $|\bk|\leq 3$,
	\begin{equation}\label{supremum-l1-norm}      \sup_{x\in B_R} \sum_{\bl\in\N_0^n} |c^{j,\bk}_{\bl}(x)|\;\theta^{\f {|\bl|}{2}}\sqrt{\bl!} 
	<\infty.
	\end{equation}
	
\end{lemma}
\begin{proof}
	Firstly for any $\theta>0, a>1$, $$   \sum_{\bl\in\N_0^n}    \sup_{x\in B_R}  |c^{j,\bk}_{\bl}(x)| \theta^{ \f{ |\bl| } {2} }\sqrt{\bl!} \le  \sup_{x\in B_R}  
	\sup_{\bl\in\N_0^n} |c^{j,\bk}_{\bl}(x)| \theta^{\f{ |\bl|} {2} }a^{ \f{| \bl|}{2}}\sqrt{\bl!}  \sum_{\bl\in\N_0^n} a^{- \f{ |\bl|}{2}}\; .$$
	Since the composition number of $|\bl|$ into $n$ blocks is at most of the order $|2\bl|^n$, we have $$ \sum_{\bl\in\N_0^n} a ^{-\f{ |\bl|}{2}}\lesssim \sum_{\ell=0}^\infty a^{- \f{\ell}{2}} (2\ell)^n<\infty.$$
	Hence \eqref{supremum-norm-1} follows from $ \sup_{x\in B_R}   \sup_{\bl\in\N_0^n} |c^{j,\bk}_{\bl}(x)|\; \hat{\Theta}^{\f{ |\bl|}{2}} a^{\f{|\bl|}{2}}\sqrt{\bl!} <\infty$, the latter clearly follows from $ \sup_{x\in B_R} \sum_{\bl\in\N_0^n} |c^{j,\bk}_{\bl}(x)|\; \theta^{\f{ |\bl|}{2}}\sqrt{\bl!} 
	<\infty$. In particular,	
	\begin{equation}\label{supremum-norm}      \sup_{x\in B_R}  \sup_{\bl\in\N_0^n} |c^{j,\bk}_{\bl}(x)| \theta^{\f{ |\bl|}{2}}\sqrt{\bl!} 
	<\infty,
	\end{equation}
	concluding the proof.
\end{proof}

We have the following limit theorem on the solution to \eqref{eq:homogenization}: 

\begin{theorem}\label{thm:homogenization}
	Consider the ODE \eqref{eq:homogenization} with \cref{cond:homogenization} for $\gamma\in\big(\frac13,\frac12\big)$ in place. Then, for each $\varepsilon>0$, there is a unique pathwise solution $(x_t^\varepsilon)_{t\in[0,T]}$. For $i,j=1,\dots,d$, let
	\begin{align*}
	\sigma_{i,j}(x,z) &\define\int_0^{\infty}\Expec{f_{i}(x,y_r)  f_{j}(z,y_0) +  f_{i}(x,y_0)  f_{j}(z,y_r)} \,dr, \\ 
	\Lambda_{i,j} (x,z)&\define\int_0^\infty\Expec{f_i(x,y_0) f_j(z,y_r)}\,dr, \\
	\Gamma(x)&\define \int_0^\infty \Expec{ (Df(\cdot,y_r)_x ( f(x,y_0))}\,dr. 
	\end{align*}
	Then the following hold:
	\begin{enumerate}
		\item Suppose that $f(\cdot,y)$ have common compact support for every $y\in \R^n$. Then there is a limiting rough path $\X\in\FC^{\gamma}\big([0,T],\C_b^{3-}(\R^d,\R^d)\big)$  such that $\X^\epsilon$ converges weakly to $\X$.	Furthermore $\X=(X, \XX_{s,t}+(t-s)\Lambda) $, where $X$ is a Gaussian field with covariance $\sigma_{i,j}(x,z)(t\wedge s)$. In particular, for any (random) initial condition $x_0$ independent of $(y_t)_{t\geq 0}$, as $\varepsilon\to 0$, the solution $x^\varepsilon$ of \eqref{eq:homogenization} converges weakly in $\C^{\gamma}\big([0,T],\R^d\big)$ to the solution of the RDE
		\begin{equation*}
			dx_t=\delta(x_t)\,d\X_t.
		\end{equation*}

		\item\label{it:sde} Suppose furthermore that $x_0\in L^\infty$ is independent of $(y_t)_{t\geq 0}$. Then the solution of \eqref{eq:homogenization} converges weakly in $\C\big([0,T],\R^d\big)$ to the unique solution of the following Kunita type It\^o SDE:
	\begin{equation}\label{eq:limiting_kunita}
		dx_t=\Gamma(x_t) dt+F(x_t, dt).
	\end{equation}
	Here  $F(x,\cdot)\define X(x)$ is a martingale with spatial parameters and with characteristics $A_{i,j}(x,z,t)=\sigma_{i,j}(x,z)t$. Furthermore, if $\phi^\epsilon$ denotes the solution flow, the $N$-point motion $(\phi^\epsilon_t(x_1), \dots,\phi^\epsilon_t(x_N))$ converges to the $N$-point motion of the limiting equation.
	\end{enumerate}
\end{theorem}

Let us record a few remarks on \cref{thm:homogenization}:

\begin{remark}
	\leavevmode
	\begin{itemize}
		\item The limiting equation \eqref{eq:limiting_kunita} is equivalent to the classical It\^o SDE
		\begin{equation*}
		dx_t= \Gamma(x_t) \, dt + \sqrt{\sigma}(x_t)\, dW_t.
		\end{equation*}
		\item \Cref{thm:homogenization} extends the results of \cite{Gehringer-Li-tagged} from product to non-product drifts and from one- to multi-dimensional environmental fast-scale noise.
		
		\item We observe that---unlike the one-dimensional work of the first two authors of this article---the limiting rough path has a non-vanishing L\'evy area. The reason for this is the non-reversibility of the Gaussian process $(y_t)_{t\geq 0}$ in dimension $n\geq 2$. Indeed,  a \textit{one-dimensional}, stationary Gaussian process is \textit{always} reversible in the sense that, for each $T>0$,
		\begin{equation*}
		(y_t)_{t\in[0,T]}\overset{d}{=}(y_{T-t})_{t\in[0,T]}.
		\end{equation*}
		In higher dimensions, stationarity of a Gaussian process is a genuinely weaker requirement than reversibility. We also note that the presence of the non-trivial area term matches the findings of \cite{Deuschel-Orenshtein-Perkowski}.
	\end{itemize}
\end{remark}

\subsection{Weak Convergence in Rough Path Spaces}

This method of lifting the RDE on $\R^d$ to a Banach space of functions has previously been successfully employed in the very nice work of Kelly and Melbourne \cite{Kelly-Melbourne}. Since there is however a minor inaccuracy in that article (see \cref{ex:not_compact}) below, we choose to recall an appropriate amount of detail of their approach. Without further notice, we shall assume the regularity assumptions of \cref{cond:homogenization} in the sequel.
Let us first elaborate on the minor inaccuracy in the work of Kelly and Melbourne, then present results leading to tightness of $\X^\epsilon$ in the rough path spaces over $\C^3_R(\R^d, \R^d)$, the space of smooth functions with compact support in $B_R$. This is \Cref{prop:weak_convergence} below which resolves a question raised in \cite{Kelly-Melbourne} and allows us to bypass the martingale method problem as well as to obtain a Kunita type SDE in the limit.

The example below shows that tightness in rough path spaces over infinite-dimensional Banach spaces is a touchy business. 

\begin{example}\label{ex:not_compact}
  When proving tightness of the driving rough path $\X^\varepsilon=(X^\varepsilon,\XX^\varepsilon)$ in \cite[Corollary 5.9]{Kelly-Melbourne} (there denoted by $\W^\varepsilon$), the authors assert that the unit ball in $\FC^{\gamma^\prime}\big([0,T],\C_b^{\alpha}(\R^d,\R^d)\big)$ is compact in the metric of $\FC^{\gamma^\prime}\big([0,T],\C_b^{\alpha}(\R^d,\R^d)\big)$ for each $\gamma<\gamma^\prime$. This, they claim, should follow by a standard Arzel\`a-Ascoli argument as in \cite[Chapter 5]{Friz-Victoir}. In the latter, however, the authors only consider rough paths with values in a \textit{finite}-dimensional Euclidean space. In fact, for any infinite-dimensional Banach space $\CX$, the embedding
$    \FC^{\gamma^\prime}\big([0,T],\CX\big)\hookrightarrow\FC^{\gamma}\big([0,T],\CX\big)$  is not compact. To see this, let us take $T=1$ and a sequence of points $\{x_n\}_{n\in\N}\subset B_{\CX}$, given by Riesz's lemma, with $|x_m-x_n|\geq\frac12$ for any $m\neq n$. 
Set $F^n_t\define t x_n$ and $   \FF^n_{s,t}\define\frac{(t-s)^2}{2}(x_n\otimes x_n)$.
   Then ${\big\{\F_n=(F^n,\FF^n)\big\}_{n\in\N}}$  is  bounded in $\FC^\gamma\big([0,T],\CX\big)$, but not compact, as
$   |\F_m-\F_n|_{\C^{\gamma}\oplus\C^{2\gamma}}\geq\frac12$ for all $m\neq n$.

\end{example}

A remedy of this lack of compactness was presented in the recent preprint \cite{Chevyrev2020}. There, the authors worked in the $p$-variation setting, but the arguments of course transfer to H\"older rough paths. We recapitulate a streamlined version of the argument in the sequel; mainly for the reader's convenience, but also to fix some notations. 

Let $R>0$. We let $\C^\alpha_R$ denote the H\"older functions supported in the ball of radius $R$, 
\begin{equation*}
  \C^\alpha_R(\R^d,\R^d)\define\big\{f\in\CX:\,\mathrm{supp}(f)\subset B_R\big\}, \qquad  \CX=\C_b^\alpha(\R^d,\R^d).
\end{equation*}

\begin{definition}\label{def:space_time_distributions}
  Let $\X=(X,\XX)$ and $\Y=(Y,\YY)$ be random variables with values in $\FC^\gamma\big([0,T],\CX\big)$. We say that $\X$ and $\Y$ are equal in \textit{finite-dimensional (space-time) distributions} if, for each $n\in\N$, $0\leq t_1<\cdots<t_m\leq T$, and $x_1^i,\dots, x_n^i\in\R^d$ ($i=1,2$),
  \begin{equation*}
	\begin{pmatrix}
	  \big(X_{t_1}(x_1^1),\XX_{t_1}(x_1^1,x_1^2)\big)\\
	  \vdots\\
	  \big(X_{t_n}(x_n^1),\XX_{t_n}(x_n^1,x_n^2)\big)
	\end{pmatrix}\overset{d}{=}
	\begin{pmatrix}
	  \big(Y_{t_1}(x_1^1),\YY_{t_1}(x_1^1,x_1^2)\big)\\
	  \vdots\\
	  \big(Y_{t_n}(x_n^1),\YY_{t_n}(x_n^1,x_n^2)\big)
	\end{pmatrix}
  \end{equation*}
  as $\Big(\R^d\oplus\big(\R^d\otimes\R^d\big)\Big)^n$-valued random variables.
\end{definition}

Our main weak convergence result of this section is as follows:
\begin{proposition}\label{prop:weak_convergence}
  Let $\gamma\in\big(\frac13,1\big]$, $\alpha>0$, and $R>0$. Let $(\X^n)_{n\in\N}$ be a sequence of random variables with values in $\FC^\gamma\big([0,T],\C_R^\alpha(\R^d,\R^d)\big)$. If there is a $p\geq 1$ such that
  \begin{equation*}
	\sup_{n\in\N}\Expec{|\X^n|^{p}_{\C^\gamma\oplus\C^{2\gamma}}}<\infty
  \end{equation*}
  and if the finite-dimensional space-time distributions of any weak limit point of $(\X^n)_{n\in\N}$ coincide, then there is a random variable $\X\in\FC^{\gamma-}\big([0,T],\C_R^{\alpha-}(\R^d,\R^d)\big)$ such that $\X^n\Rightarrow\X$ weakly.
\end{proposition}

\begin{remark}
  \Cref{prop:weak_convergence} resolves the question raised in \cite[Remark 5.14]{Kelly-Melbourne}. In fact, by employing it, one could bypass the invocation of the martingale problem used in \cite[Section 6]{Kelly-Melbourne} in order to characterize the limiting equation. 
\end{remark}

\begin{proof}[of \cref{prop:weak_convergence}]
  Fix $\gamma^\prime<\gamma$ and $\alpha^\prime<\alpha$. Let $\varepsilon>0$. Then  $(\X^n)_{n\in\N}$ is tight in the space $\FC^{\gamma^\prime}\big([0,T],\C^{\alpha^\prime}_R(\R^d,\R^d)\big)$. Indeed, there is an $M>0$ such that  $ \P\big(\X^n\notin K_M\big)\leq\varepsilon$
  where   $$    K_M\define\big\{\X\in\FC^{\gamma}\big([0,T],\C^{\alpha}_R(\R^d,\R^d)\big):\,|\X|_{\C^\gamma\oplus\C^{2\gamma}}\leq M\big\}.
$$
By  \cref{lem:compactness} below this set is relatively compact as a subset of  $\FC^{\gamma'}\big([0,T],\C^{\alpha'}_R(\R^d,\R^d)\big)$.
  
 It remains to show that the finite-dimensional space-time distributions uniquely characterize the limit points of $(\X^n)_{n\in\N}$. 
 Let $\X$ and $\tilde\X$ be limit points of $(\X^n)_{n\in\N}$. By the Portmanteau theorem, the laws of both $\X$ and $\tilde\X$ are Radon measures. A  compactification argument (see e.g. \cite[Exercise 7.14.79]{Bogachev2007}) shows that these two measures coincide, provided we can exhibit a test set  $\CF$ of bounded continuous functions $f:\FC^{\gamma^\prime}\big([0,T],\C_R^{\alpha^\prime}(\R^d,\R^d)\big)\to\mathbf{C}$ such that
  \begin{itemize}
	\item $\CF$ separates points on $\FC^{\gamma^\prime}\big([0,T],\C_R^{\alpha^\prime}(\R^d,\R^d)\big)$,
	\item $\CF$ is closed under pointwise multiplication,
	\item $\1\in\CF$, where $\1$ is the function with constant value $1$,
	\item $\Expec{f(\X)}=\Expec{f(\tilde\X)}$ for all $f\in\CF$.
  \end{itemize}
  We choose $\CF$ as the family of characteristic functionals furnished by the space-time evaluations of \cref{def:space_time_distributions}:
  \begin{align*}
	 \CF\define\bigg\{(X,\XX)\mapsto&\exp\left(i\sum_{j=1}^n\pi_j\Big(X_{t_j}(x_j^1),\XX_{t_j}(x_j^1,x_j^2)\Big)\right):\\
	 & n\in\N,\,t_j\in[0,T],\,x^{1,2}_j\in\R^d,\,\pi_j\in\L\Big(\R^d\oplus\big(\R^d\otimes\R^d\big),\R\Big)\bigg\}.
  \end{align*} 
One verifies that $\CF$ satisfies the requirements above, whence $\X\overset{d}{=}\tilde\X$, as required.
\end{proof}

\begin{lemma}\label{lem:compactness}
  Let $R>0$, $\frac13<\gamma^\prime<\gamma\leq\frac12$, and $0<\alpha^\prime<\alpha$. Then the embedding
  \begin{equation}\label{eq:embedding}
	\FC^\gamma\big([0,T],\C_R^\alpha(\R^d,\R^d)\big)\hookrightarrow\FC^{\gamma^\prime}\big([0,T],\C_R^{\alpha^\prime}(\R^d,\R^d)\big)
  \end{equation}
  is compact. \end{lemma}

\begin{proof}
The embedding \eqref{eq:embedding} is certainly continuous. To see that it is actually compact, it is enough to show that the set 
$
	K_1\define\big\{\X\in\FC^{\gamma^\prime}\big([0,T],\C^{\alpha^\prime}_R(\R^d,\R^d)\big):\,|\X|_{\C^\gamma\oplus\C^{2\gamma}}\leq 1\big\}$
  is relatively compact.
 
  We shall make use of a general version of the Arzel\`a-Ascoli theorem recalled in \cref{prop:ascoli} after the proof.
   Let $(\X^n)_{n\in\N}$ be a sequence in $K_1$. We need to show that both $(X^n)_{n\in\N}$ and $(\XX^n)_{n\in\N}$ are relatively compact in the spaces $\C^{\gamma^\prime}\big([0,T],\C^{\alpha^\prime}_R(\R^d,\R^d)\big)$ and $\C^{2\gamma^\prime}\big(\Delta_T,\C^{\alpha^\prime}_R(\R^d,\R^d)\otimes\C^{\alpha^\prime}_R(\R^d,\R^d)\big)$, respectively. Since the arguments are similar, we only detail the relative compactness of $(\XX^n)_{n\in\N}$.

  By the algebraic constraint \eqref{eq:chen} and a straight-forward interpolation estimate (see e.g. \cite[Exercise 2.9]{Friz-Hairer}), it is enough to show that $(\XX^n_{0,\cdot})_{n\in\N}$ is relatively compact in $\C\big([0,T],\C^{\alpha^\prime}_R(\R^d,\R^d)\otimes\C^{\alpha^\prime}_R(\R^d,\R^d)\big)$.  
  First note that this family is  equicontinuous. Indeed, again by Chen's relation,
  \begin{equation*}
	\|\XX^n_{0,t}-\XX^n_{0,s}\|_{\C^{\alpha^\prime}\otimes\C^{\alpha^\prime}}\leq\|\XX^n_{s,t}\|_{\C^{\alpha^\prime}\otimes\C^{\alpha^\prime}}+|X^n_s|_{\C^{\alpha^\prime}}|X^n_{s,t}|_{\C^{\alpha^\prime}}
  \end{equation*}
  for all $0\leq s\leq t\leq T$. We are thus left to show that, for each $t\in[0,T]$, the set $\big\{\XX^n_{0,t}:\,n\in\N\big\}$ is relatively compact in $\C_R^{\alpha^\prime}(\R^d,\R^d)\otimes\C_R^{\alpha^\prime}(\R^d,\R^d)$ in order to conclude with \cref{prop:ascoli} below. For this it is enough to note the compact embedding
  \begin{equation*}
	\C^{\alpha}\big(\overline{B}_R\times\overline{B}_R,\R^d\otimes\R^d\big)\hookrightarrow\C^{\alpha^\prime}\big(\overline{B}_R\times\overline{B}_R,\R^d\otimes\R^d\big),
  \end{equation*}
  which  in turn again follows from the Arzel\`a-Ascoli theorem---but this time in the spatial coordinate. 
\end{proof}

The following version of the Arzel\`a-Ascoli theorem, employed in the previous proof, can be found in multiple places, see e.g. {\cite[Theorem 7.17]{Kelley1975}}:
\begin{proposition}[Arzel\`a-Ascoli]\label{prop:ascoli}
  Let $X$ be a compact metric space and $Y$ be a metric space. Let $\C(X,Y)$ be the space of continuous mappings $f:X\to Y$, equipped with the uniform topology. Then a set $K\subset\C(X,Y)$ is relatively compact if and only if
  \begin{itemize}
	\item the set $\{f(x):\,f\in K\}$ is relatively compact in $Y$ for each $x\in X$ and
	\item the set $K$ is equicontinuous, i.e., for all $\varepsilon>0$, there is a $\delta>0$ such that $d_Y\big(f(x),f(y)\big)\leq\varepsilon$ for all $f\in K$, provided that $d_X(x,y)\leq\delta$.
  \end{itemize}
\end{proposition}

\subsection{Proof of \cref{thm:homogenization}}

We can now conclude the proof our homogenization result:
\begin{proof}[of \cref{thm:homogenization}]
We first assume that $f(\cdot,y)$ and $g$ are supported in $\overline{B}_R$ for each $y\in\R^n$. Owing to \cref{prop:rde_existence,lem:ode}, it is enough to show that
  \begin{equation}\label{eq:weak_to_show}
	\X^\varepsilon\Rightarrow \X
  \end{equation}
  in $\FC^\gamma\big([0,T],\C^{3-}_R(\R^d,\R^d)\big)$.
  We wish to employ \cref{prop:weak_convergence}. To see that
  \begin{equation}\label{eq:tightness}
	\sup_{\varepsilon\in(0,1]}\Expec{\big|X^\varepsilon|_{\C^\gamma}}<\infty,
  \end{equation}
  we argue as in \cref{lem:lp_holder}: First notice that, for any $i=1,\dots,d$ and any $p\geq 1$,
  \begin{align*}
  	\left\|\sup_{x\in\overline{B}_R}\frac{1}{\sqrt{\varepsilon}}\int_s^t f_i(x,y_{\frac{u}{\varepsilon}})\,du\right\|_{L^p(\Omega)}&=\frac{1}{\sqrt{\varepsilon}}\left\|\sup_{x\in\overline{B}_R}\sum_{\bl\in\N_0^n}c_\bl^i(x)\int_s^t  H_\bl(y_{\frac{u}{\varepsilon}})\,du\right\|_{L^p(\Omega)}\\
  	&\leq\frac{1}{\sqrt{\varepsilon}}\sum_{\bl\in\N_0^n}\sup_{x\in\overline{B}_R}\big|c_\bl^i(x)\big|\left\|\int_s^t H_\bl(y_{\frac{u}{\varepsilon}})\,du\right\|_{L^p(\Omega)}.
  \end{align*}
  By hypercontractivity and our assumption $c_{\bl}(x)=0$ for all $x\in\R^d$ whenever $|\bl|<\inf_x\herm\big(f(x,\cdot)\big)$, we get as in \eqref{eq:lp_hypercontractivity}
  \begin{equation*}
  	\left\|\sup_{x\in\overline{B}_R}\frac{1}{\sqrt{\varepsilon}}\int_s^t f_i(x,y_{\frac{u}{\varepsilon}})\,du\right\|_{L^p(\Omega)}\leq\sqrt{|t-s|}\sum_{\bl\in\N_0^n}\sup_{x\in\overline{B}_R}\big|c_\bl^i(x)\big|\sqrt{\bl!}\hat{\Theta}^{\f{| \bl | }{2}}(2n-1)^{\frac{|\bl|}{2}}(p-1)^{\frac{|\bl|}{2}}.
  \end{equation*}
  The series on the right-hand side is finite by the assumption \eqref{supremum-norm-1}. By \cref{lem:hermite_rank_derivative}, we know that the derivatives of $f$ have (at least) the same Hermite rank. Hence, we can similarly prove that
  \begin{equation*}
   	\left\|\sup_{x\in\overline{B}_R}\frac{1}{\sqrt{\varepsilon}}\int_s^t D^k_x f_i(x,y_{\frac{u}{\varepsilon}})\,du\right\|_{L^p(\Omega)}\lesssim\sqrt{|t-s|},\qquad k=1,2,3
  \end{equation*}
  and \eqref{eq:tightness} follows by Kolmogorov's continuity theorem. Arguing as in \cref{lem-tightness-iterated-hermite-poly} also shows that $\sup_{\varepsilon\in(0,1]}\Expec{|\XX^\varepsilon|_{\C^{2\gamma}}|}<\infty$. Then 
\Cref{thm-A} and \cref{rem:stratonovich_correction} show that $\X^\epsilon$ converges to $\X$ with $\XX_{s,t}(x,z)=(\Upsilon\otimes\Upsilon)\WW_{s,t}^{\text{It\^o}} (x,z)+\Lambda(x,z) (t-s)$.

It remains to prove statement \ref{it:sde}. Observe that $X$ is a Gaussian process  with covariance $\sigma_{i,j}(x,z)(s\wedge t)$ and 
\begin{equation*}
	G(x,t)\define X_t(x)+\Gamma t
\end{equation*}
 is a semi-martingale with spatial parameters and characteristics 
 \begin{equation*}
 	A_{i,j}(x,z,t)\define \sigma_{i,j}(x,z)t \qquad \hbox{ and } \qquad \Gamma(x)t.
 \end{equation*}
The characteristics are sufficiently regular so that  the Kunita type equation 
\begin{equation}\label{Kunita-limit}
dx_t=F(x_t, dt) +\Gamma(x_t)
\end{equation}
is well posed. The regularity of  $A_{i,j}$ and $\Gamma$ comes from the uniform correlation decay  assumption $\inf_x\herm\big(f(x,\cdot)\big)>\f 2 \beta$. For example, the functions  
$\Expec{ D (f_{i}(x,y_s)  f_{j}(z,y_0) +  f_{i}(x,y_0)  f_{j}(z,y_s))}$  are absolutely integrable  in $s$ on $[0, \infty)$, consequently
 $$\sigma_{i,j}(x,z)=\int_0^{\infty}\Expec{f_{i}(x,y_r)  f_{j}(z,y_0) +  f_{i}(x,y_0)  f_{j}(z,y_r)} \,dr$$
 is $\C_b^3$ in both variables and jointly continuous in $(x,z)$. The same argument applies to $\Gamma$.

Furthermore, by the theory for SDEs driven by semi-martingales with spatial parameters \cite{Kunita}, 
there is a unique Brownian flow $\phi_{s,t}(x)$ to the equation (\ref{Kunita-limit}) and for each $s,x$,
$$\phi_{s,t}(x)-x-\int_s^t \Gamma(\phi_{s,r})dr$$
is a square integrable martingale with $\lim_{h\to 0} \E [ (\phi_{t, t+h}(x)-x) (\phi_{t, t+h}(y)-y)^T]=\sigma_{i,j}(x,y) t$.
The RDE is `equivalent' to the Kunita type SDE which can be seen from the Riemann sum approximation for integration with respect to the semi-martingale $G$ : 
$$\int_0^t G(x_s, ds)\sim \sum_{[u,v]\subset {\mathcal P}} G(x_u, u)-G(x_v, v)$$
where ${\mathcal P}$ denotes  a partition of $[0, t]$. 
 On the other hand, the rough integral, 
$$\int_0^t \delta(x_s) d\X_s\sim  \sum_{[u,v]\subset {\mathcal P}} \delta(x_u)(X_u-X_v)+ \delta^\prime (x_u) (\XX^{\textup{It\^o}}_{u,v}+\Lambda (u-v)),$$
where the prime on $\delta(x_t)$ denotes its Gubinelli derivative, which is $(D\delta)_{x_u} (\delta(x_u)\cdot) (\cdot)$ 
when applied to $\Lambda(\cdot,\cdot)$, is:
$$\Gamma_j(x)=\int_0^\infty \E[ \Braket{D f_j (\cdot, y_r)_x, f(x,y_0)}] \,dr.$$
It is then routine to verify that $\phi_t(x)$ is the solution to the rough differential equation 
\begin{equation}\label{eq:ito_limit-2}
dx_t=\delta(x_t) d\X_t+\Lambda(x_t) dt.
\end{equation}
The latter equation is also well posed.	
	
We then fix $R>0$ and let $\eta_R:\R^d\to\R_+$ be a non-negative smooth function with $\eta_R=1$ on $\overline{B}_R$ and $\eta_R=0$ on $B_{2R}^c$. Set $f_R(x,y)\define \eta_R(x)f(x,y)$. By the first part of the theorem, we know that there is an $\X^R=(X^R,\XX^R)\in\FC^{\gamma}\big([0,T],\C_R^{3-}(\R^d,\R^d)\big)$ such that $x^{R,\varepsilon}\Rightarrow x^R$, where
\begin{align*}
	dx_t^{R,\varepsilon}&=\frac{1}{\sqrt{\varepsilon}}f_R(x_t^{R,\varepsilon},y_t)\,dt, \\
	dx_t^{R}&=\delta(x_t^{R})\,d\X^R_t.
\end{align*}
As before, the convergence is simultaneous for any finite number of initial conditions.
Mimicking that for $G$, let $G_R=\X^{R}+\Gamma_R t$ denote the spatial semi-martingale with characteristics $\sigma^R_{i,j}(x,y)$ and $\Gamma_R t$. Then for each initial condition $x_0$, the solutions of the Kunita type SDEs
\begin{equation*}
	dx_t^{R}=G_R(x_t^{R},dt)
\end{equation*}
converge weakly to the solution of \eqref{eq:ito_limit-2} with the same initial distribution. We have applied  \cite[Thm. 5.2.1]{Kunita},  it is trivial to verify 
the conditions there: $\sup _x |D^a  D^b \sigma_R(x,y)|\big|_{y=x}$, where $a,b\in \{0,1\}$ and $\sup_x |D_x \Gamma_R|$ are uniformly bounded, also the characteristics are in the required class $\tilde C_b^\alpha$ for some $\alpha>0$, c.f. \cite{Kunita},  and converge uniformly in $(x,y)$ on compact sets. The convergence is in the sense that, given any initial data $(x_0^i)$, $i=1, \dots, N$, the $N$-point motion converges.
 
Finally, we show the weak convergence of $x^{\epsilon}\Rightarrow x$. By the Portemanteau theorem,  it is equivalent to showing that 
$\limsup_{\varepsilon\to 0}\P(x^\varepsilon\in A)\leq\P(x\in A)$  for any closed set $A\subset\C\big([0,T],\R^d\big)$.
Since $x^{R, \epsilon}\Rightarrow x^R$ weakly in $\C\big([0,T],\R^d\big)$, the Portemanteau theorem gives the estimate:
  \begin{align}
  	\limsup_{\varepsilon\to 0}\P(x^\varepsilon\in A)&\leq\limsup_{\varepsilon\to 0}\P(x^{R,\varepsilon}\in A)+\limsup_{\varepsilon\to 0}\P\big(|x^{R,\varepsilon}|_\infty>R\big)\nonumber\\
  	&\leq \P(x^{R}\in A)+\P(|x^R|_\infty\geq R).\label{eq:localization}
  \end{align}
Hence, $\limsup_{R\to\infty}\P(x^{R}\in A)\leq\P(x\in A)$. Note also that $\P(|x^R|_\infty\geq R)=\P(|x|_\infty\geq R)$ for $R>\|x_0\|_{L^\infty}$. Markov's inequality gives that $\P(|x|_\infty\geq R)\to 0$ as $R\to\infty$. Consequently, sending $R\to\infty$ in \eqref{eq:localization}, we have proven that
$
  	\limsup_{\varepsilon\to 0}\P(x^\varepsilon\in A)\leq\P(x\in A)$
  which concludes the proof for the weak convergence of $x^{\epsilon}\Rightarrow x$. The proof for the convergence of the $N$-point motion is an easy adaption of the above, as we have the $N$-point convergence for both $x^{\epsilon, R}$ as $\epsilon \to 0$ and for $x^R$ as $R\to \infty$.
\end{proof}

{
\normalem
\small
\bibliographystyle{alpha}
\bibliography{drafts.bib}
}

\end{document}